\documentclass[draft]{article}
\usepackage{amsmath,amsfonts,latexsym, amssymb}
\usepackage{dsfont}
\usepackage{stmaryrd}
\usepackage{enumerate}

\usepackage{color}

\newcommand{\ii}{{\mathrm{i}}}
\DeclareMathOperator{\OO}{O}
\newcommand{\var}{{\mathrm{var}}}
\newcommand{\wt}{\widetilde}
\newcommand{\wh}{\widehat}

\renewcommand{\epsilon}{\varepsilon}
\newcommand{\e}{\varepsilon}
\newcommand{\eps}{\varepsilon}
\newcommand{\pt}{\partial}
\newcommand{\rd}{{\rm d}}
\newcommand{\Id}{{\rm Id}}
\newcommand{\bR}{{\mathbb R}}

\DeclareMathOperator{\Edg}{Edg}
\DeclareMathOperator{\Ext}{Ext}
\DeclareMathOperator{\Int}{Int}

\renewcommand{\Im}{\mbox{Im}}
\newcommand{\bm}{{\bf{m}}}

\newcommand{\bx}{{\bf{x}}}
\newcommand{\by}{{\bf {y} }}

\newcommand{\bw}{{\bf{w}}}
\newcommand{\bz}{{\bf {z}}}

\newcommand{\bla}{\mbox{\boldmath $\lambda$}}

\newcommand{\be}{\begin{equation}}
\newcommand{\ee}{\end{equation}}

\newcommand{\g}{{\sigma}}

\newcommand{\la}{\lambda}
\newcommand{\Om}{{\Omega}}
\newcommand{\om}{{\omega}}

\renewcommand{\th}{\theta}

\newcommand{\cL}{{\mathcal L}}

\newcommand{\cG}{{\mathcal G}}

\newcommand{\E}{{\mathbb E }}
\newcommand{\R}{{\mathbb R }}

\renewcommand{\P}{{\mathbb P}}

\newcommand{\ind}{{\,\mathrm{d}}}

\newtheorem{theorem}{Theorem}
\newtheorem{corollary}[theorem]{Corollary}
\newtheorem{lemma}[theorem]{Lemma}
\newtheorem{proposition}[theorem]{Proposition}

\newtheorem{definition}[theorem]{Definition}
\newcommand{\qed}{\hfill\fbox{}\par\vspace{0.3mm}}
\newenvironment{proof}{{\bf Proof.}} {\hfill\qed}

\numberwithin{theorem}{section}

\usepackage{geometry}     
\usepackage{graphicx}

\numberwithin{equation}{section}
\def\cal{}
\def\RR{{\mathbb R}}

\def\NN{{\mathbb N}}
\def\CC{{\mathbb C}}

\def\cH{{\mathcal H}}

\def\W2{W^{1,2}({\cal O}(M))}

\def\1half{\frac{1}{2}}

\title{Universality of General $\beta$-Ensembles}

\date{Feb 5, 2012}

\author{Paul Bourgade${}^1$\quad
L\'aszl\'o Erd\H os${}^2$\thanks{Partially supported
by SFB-TR 12 Grant of the German Research Council} \quad
Horng-Tzer Yau${}^1$\thanks{Partially supported
by NSF grants DMS-0757425, 0804279}
 \\\\
Department of Mathematics, Harvard University\\
Cambridge MA 02138, USA \\   bourgade@math.harvard.edu \quad
 htyau@math.harvard.edu ${}^1$ \quad  \\ \\
Institute of Mathematics, University of Munich, \\
Theresienstrasse 39, D-80333 Munich, Germany \\ lerdos@math.lmu.de ${}^2$
\\}

\begin{document}

\maketitle

\vspace{1cm}

\begin{abstract}
We prove the universality of the $\beta$-ensembles with convex
analytic potentials and for any $\beta>0$, i.e. we show
that the spacing distributions of  log-gases  at any inverse
temperature $\beta$ coincide with those   of
the Gaussian $\beta$-ensembles.
\end{abstract}

\vspace{1.5cm}


\medskip

\medskip


\medskip

\newpage

\section{Introduction}

The central concept of the random matrix theory  as envisioned by E.  Wigner
  is the general hypothesis that
the distributions of eigenvalue spacings
 of  large complicated quantum systems are universal
in the sense that they depend only on the symmetry classes
of the physical systems but not  on other  detailed structures.
The simplest case for this hypothesis is for  ensembles of large but finite dimensional
 matrices. The general hypothesis in this setting thus asserts that
 the eigenvalue spacing distributions of random matrices should be independent of
 the probability distribution of the ensemble, up to scaling.
This is generally referred to as the universality of
 random matrices.
In this paper we will focus only on the bulk behavior
i.e., on eigenvalue distribution in the interior of the spectrum,
although similar questions
regarding the edge distribution are also important.

 Over the past two decades,  spectacular  progress (see, e.g.,
\cite{BI, DKMVZ1, DKMVZ2, PS:97, PS, DG, Lub} and \cite{AGZ, De1, DG} for a review)
 on  bulk  universality
was made for classical invariant ensembles, i.e., matrix models with probability  measure given by $e^{- N \beta  { \rm Tr} V(H)/2 }/Z$
where $N$ is the size of the matrix $H$, $V$ is a real valued potential and $Z$
is the normalization.
It is well-known that the probability distribution of the ordered eigenvalues of $H$
 on
the simplex determined by $\lambda_1\leq \dots\leq \lambda_N$
 is given by
\begin{equation}\label{01}
\mu^{(N)}\sim e^{- \beta N \cH},
\quad
\cH =   \sum_{k=1}^N  \frac{1}{2}V(\lambda_k)-
\frac{1}{N} \sum_{1\leq i<j\leq N}\log (\lambda_j-\lambda_i) ,
\end{equation}
where the parameter  $\beta=1,2,4$ is determined by the symmetry type of the matrix,
corresponding respectively to the classical orthogonal, unitary  or symplectic ensemble.
With $\beta$ taking these special values, the correlation functions can be explicitly expressed
in terms of polynomials orthogonal to the measure $e^{- \beta V(x)/2}$.
 Thus the analysis of the correlation functions
relies heavily on the asymptotic properties of the  corresponding orthogonal polynomials. In the pioneering  work
of Gaudin, Mehta  and Dyson (see \cite{M} for a review), the potential
$V$ is the quadratic polynomial $V(x) = x^2$ and the orthogonal polynomials are the Hermite polynomials
for which asymptotic properties are well-known.
The major input of  the recent work
is the asymptotic analysis of  the orthogonal polynomials w.r.t. the measure $e^{- \beta V(x)/2}$ for general
classes of potentials.  The formulas  for orthogonal and symplectic cases, i.e.,  $\beta=1, 4$,  are much more difficult to use than
the one for the unitary case. While universality for $\beta=2$ was proved for very general potential,
the best results  for $\beta = 1, 4$ \cite{DG, KS, Sch} are still restricted to analytic $V$ with additional conditions.

For non-classical values of $\beta$, i.e., $\beta \not \in \{1, 2, 4\}$,
one can still consider the measure (\ref{01}), but
there is no simple expression of the correlation functions in
terms of orthogonal polynomials. Furthermore, the
 measure \eqref{01} does  not arise from  mean-field type matrix models
like Wigner matrices with independent entries.
  Nevertheless,  $\mu$ is a Gibbs measure of particles
in $\bR$  with a logarithmic interaction, where the parameter $\beta$
is interpreted as the inverse temperature and a priori can be an arbitrary positive number.
These measures are called general $\beta$-ensembles. We will often
refer to the variables $\lambda_j$ as particles or points
and the system is called log-gas.
It was proved   \cite{DumEde} that in the Gaussian case, i.e.,
when $V$ is quadratic, the measure \eqref{01} describes eigenvalues of
tri-diagonal matrices. This observation allowed one to establish
detailed properties, including the local
spacing distributions of the Gaussian $\beta$-ensembles \cite{VV}.

Gibbs measures in the continuum with long range or singular interactions are notoriously hard to analyze since
they are very far from the perturbative regime. For non-classical values of $\beta$, and
if we are not in the Gaussian case $V(\lambda)=\lambda^2$,
no  simple explicit formula
is known  to express  the correlation functions in terms of orthogonal polynomials,
and one cannot rely on any explicit known matrix model.
In this paper we undertake
the direct analysis of the Gibbs measure and we
prove the universality for invariant models for any $\beta > 0$.
In other words, we will prove
that the local spacing distributions of \eqref{01} are independent of the potential $V$ for certain class of $V$.
There are two major ingredients in  our new approach.

\medskip
\noindent
{\it Step 1.  Uniqueness of local Gibbs measures with logarithmic interactions.}
 The main result in this step  asserts that if the particles are not too far
 from their classical locations then
the spacing distributions are given by the corresponding Gaussian ones (We will take the uniqueness of the spacing
distributions as our definition of the uniqueness of Gibbs state).
More precisely, denote by $\rho$ the limiting density of the particles  under  the measure $\mu^{(N)}$ \eqref{01} as $N\to\infty$.
Let $\gamma_j =\gamma_{j,N}$ denote the location of the $j$-th point under $\rho$, i.e.,
 $\gamma_j$ is defined by
\be\label{def:gamma}
 N \int_{-\infty}^{\gamma_j} \rho (x) \rd x = j, \qquad 1\leq j\le N. \quad
\ee
We will call $\gamma_j$ the {\it classical location} of the $j$-th particle. The basic assumption is the following:

\medskip
\noindent
{\bf Assumption A.} For some  $  {\frak b} < \frac 1 {38} $ and any $\alpha>0$, there exists $\epsilon_0>0$ such that
 \be\label{assA}
\P_{\mu^{(N)}}  ( |\lambda_k-\gamma_k|\le N^{-1 + {\frak b}} ) \ge 1 -  \exp(-N^{\e_0})
\ee
for large enough $N$ and any $k \in [\alpha N, (1-\alpha) N]$.

\medskip
\noindent
Under this assumption
(under some minor and easily verifiable assumptions near the edges of the
limiting measure),
we will prove that the  spacing distributions of $\mu$ are given
 by the corresponding Gaussian model
with $V(x) = x^2$.   We will use the Gaussian case as our reference
ensemble only for the convenience of definiteness.   In fact, no detailed
properties of the Gaussian measures  are used in the proof and any
other reference ensemble would have worked as well.
Furthermore, in this step
we make no  assumption on the convexity of $V$,  which is needed in the next step.

\medskip
\noindent
{\it Step 2.  Particle  location estimate.}
The second step is to  verify  Assumption A. For non-classical $\beta$, Assumption A
is only proved for $\frak b$ near one \cite{Joh, PS, KS}
 for analytic potential $V$ under certain constraint.
This is far from sufficient to complete Step 1.  We will prove Assumption A
for all $\beta > 0$ under the assumption that $V$ is convex and analytic.  Our method uses the following three ideas:  (1)  The analysis of
the loop equation
 in \cite{Joh, KS, Sch} to control the density.
(2) The  logarithmic Sobolev
inequality guaranteed by the  convexity of $V$.
(3) A multiscale analysis of the probability measures
of invariant ensembles.
We note that the assumption of analyticity on $V$ is needed only for using
the loop equation in  (1).

The basic idea of our proof is to use the following tool from \cite{ESY4}: For two probability
measures $\mu$ and $\om$ define the Dirichlet form by
\[
  D(\mu\mid \om) : =
 \frac{1}{2 N} \int \Big|\nabla \sqrt{ \frac{\rd\mu}{\rd\om}}\Big|^2\rd \om.
\]
Then  the difference of the local spacing distributions
of the two measures is negligible  provided that
the Dirichlet form per particle is sufficiently small in the large $N$ limit \cite{ESY4}.
Notice that if we used  the relative entropy of the two measures, then the uniqueness
 of the Gibbs measures would require the total entropy, which is an extensive quantity, to  be small.
To apply this Dirichlet form  inequality,
we first localize the measure  by fixing $\lambda_j$
for $j$ outside, say, the interval $[L+1, L + K]$ for $L$ in the bulk and $K = N^k$ for
some $k > 0$.
We will call these data of $\lambda_j$ outside the interval
 $[L+1, L + K]$ the boundary condition.
We then compare this measure to a local Gaussian $\beta$-ensemble
 with a fixed boundary condition by showing that
 the Dirichlet form per particle of these two measures is small
 for typical boundary conditions
w.r.t. $\mu$.

Our  approach shares  some philosophy from the recent
 method on the universality of Wigner matrices \cite{ESY4, ESYY}.
In this approach,  the key  condition to establish is

\bigskip
{\bf Assumption III.} There exists an  ${\frak a} > 0$ such that we have
\be
 \E_{\mu_W}   \frac{1}{N}\sum_{j=1}^N(x_j-\gamma_j)^2
  \le CN^{-1-2{\frak a}}
\label{assum3}
\ee
with a constant $C$ uniformly in $N$. Here $\mu_W$ is the law given by the Wigner ensemble.

Under this assumption,  a strong estimate on the local ergodicity of
Dyson Brownian motion (DBM)  was established  in \cite{ESY4, ESYY}.
DBM \cite{Dy} establishes a dynamical interpolation between Wigner matrices and
the invariant equilibrium measure $\mu$.
This estimate then implies the universality of Wigner matrices. Thus the main task
in proving the universality of Wigner matrices is reduced to verifying Assumption III.

There are several similarities between the method used for the universality of
Wigner matrices \cite{ESY4, EYYrigi} and the current proof for $\beta$-ensembles:
  (i)  Both rely on crude estimates such as \eqref{assA} and
  \eqref{assum3} on the location of the
eigenvalues to establish the local spacing distributions are the same
 as in the Gaussian cases.
(ii) Both use estimates on the Dirichlet form to identify the
 local spacing distributions. (iii) The main model dependent argument
is to prove these crude bounds on the eigenvalues.
The precision of these a-priori estimates
on the eigenvalues
is weaker than the local spacing, but better than previously known results on eigenvalue locations:
we have to develop
new methods to prove \eqref{assA} and
  \eqref{assum3}.

There are, however, substantial differences between the proofs of universalities for Wigner and $\beta$-ensembles.
First, since the $\beta$-ensembles  are already in equilibrium,
 there is no dynamical relaxation mechanism to exploit and the
 local statistics need to be identified directly without dynamical argument.
Second, we obtain the crude estimate \eqref{assA}
by a method  completely different from the Wigner matrices, as there is no underlying
matrix ensemble with independent entries to analyze. The accuracy result we obtain by this new method is actually optimal,
i.e. \eqref{assA} will be shown to hold for any ${\frak b}>0$.

\section{Statement of the main result}

Consider a probability measure
\begin{equation}\label{eqn:measure}
\mu^{(N)}_{\beta, V}=\mu^{(N)}(\rd\lambda)=\frac{1}{Z_N}
\prod_{1\leq i<j\leq N}|\lambda_i-\lambda_j|^\beta\prod_{k=1}^N e^{-N\frac{\beta}{2}V(\lambda_k)}
\rd\lambda_1\dots\rd\lambda_N,
\end{equation}
where $\lambda=(\lambda_1,\dots,\lambda_N)$, $\lambda_1\leq \dots\leq \lambda_N$.
Here the inverse temperature satisfies $\beta>0$ and the external potential $V$ is any
convex real analytic function   in $\RR$, and such that
\begin{equation}\label{eqn:LSImu}
\varpi =\frac{\beta}{2} \inf_{x\in\RR}V''(x) > 0.
\end{equation}
For such a convex potential, as noted in the next section the equilibrium measure, denoted by $\rho(s)\rd s$, is supported on a single interval $[A,B]$.
In the following, we omit the superscript $N$ and we will write $\mu$ for $\mu^{(N)}$.
We will use $\P_\mu$ and $\E_\mu$ to denote the probability and the
expectation with respect to $\mu$.

The  Gaussian case corresponds to
$V(\lambda)=\lambda^2$; the expectation with respect to this
Gaussian measure will be denoted by $\E_{\rm{Gauss}}$, and the equilibrium measure is
known to be
$$
\rho_{sc}(E):=\frac{1}{2\pi}\sqrt{(4-E^2)_+},
$$
the semicircle density.  The Gaussian case includes
the classical GUE, GOE and GSE ensembles for the special choice of $\beta=1,2,4$, but
our result holds for all $\beta > 0$.

Now we state our  main theorem which will be proven
at the end of Section~\ref{sec:loceq}:

\begin{theorem}\label{thm:Main} Assume
$V$ is any
real analytic function with 
$\inf_{x\in\RR}V''(x) > 0$.
Let $\beta> 0$.
 Consider the $\beta$-ensemble $\mu=\mu_{\beta, V}$.
Let $G:\RR\to \RR$ be a smooth, compactly supported function.
Let $E\in (A,B)$ lie in the interior of the support of $\rho$, and
similarly let $E'\in (-2,2)$ be inside the support of $\rho_{sc}$.  Define $L$ and $L'$ by
$$
  \frac{L}{N} = \int_{A}^{E} \rho(x) \rd x , \qquad
\frac{L'}{N} = \int_{-2}^{E'} \rho_{sc}(x) \rd x.
$$
Fix a parameter $K=N^k$ where $0<k\le \frac{1}{2}$ is an arbitrary constant. Let
 $I$ and $I'$  be two intervals of natural numbers,
$I=[L+1, L+K]$, $I'=[L'+1, L'+K]$
 with length $K=|I|$.
Then
\be
\lim_{N\to\infty}\Bigg| \E_\mu \frac{1}{K\rho(E)}\sum_{i\in I} G\Big(
 \frac{N(\la_i-\la_{i+1})}{\rho(E)}\Big)
 - \E_{\rm Gauss}
\frac{1}{K\rho_{sc}(E')}
\sum_{i\in I'} G\Big( \frac{N(\la_i-\la_{i+1})}{\rho_{sc}(E')}\Big)\Bigg|=0,
\label{eq:Main}
\ee
i.e. the appropriately normalized  particles gap distribution of  the
measure $\mu_{\beta, V}$ at the level $E$ in the bulk of the limiting density  asymptotically coincides with that
for the Gaussian case and it is independent of the value of $E$ in the bulk. In particular the gap distribution  is universal.
\end{theorem}

{\it Remark.} The same result (with the same proof) holds for higher order correlation functions
of particles gaps. More precisely,
fix $n\ge 1$ and an array of positive  integers, $\bm = (m_1, m_2, \ldots, m_n)\in
\NN^n_+$.
Let  $G:\bR^n\to\bR$ be a bounded smooth function with compact
support and we
 define
\be
  \cG_{i,\bm}(\bla) :=
\frac{1}{\rho(E)^n}G\Big( \frac{N(\la_i-\la_{i+m_1})}{\rho(E)} \, ,
\frac{N(\la_{i+m_1}-\la_{i+m_2})}{\rho(E)} \, ,
\ldots, \frac{N(\la_{i+m_{n-1}}-\la_{i+m_n})}{\rho(E)}\Big).
\label{cG}
\ee
Then, under the conditions of Theorem~\ref{thm:Main} and
 using its  notations,
 we have
\be
\lim_{N\to\infty}\Bigg| \E_\mu \frac{1}{K}\sum_{i\in I} \cG_{i,\bm}(\bla)
 - \E_{\rm Gauss}
\frac{1}{K}
\sum_{i\in I'} \cG_{i,\bm}'(\bla)\Bigg|=0,
\label{eq:Mainm}
\ee
where $\cG_{i, \bm}'$ is defined exactly as $\cG_{i,\bm}$ but
$\rho(E)$ is replaced with $\rho_{sc}(E')$.

The limit \eqref{eq:Mainm} can be reformulated as
 the convergence of the  correlation functions.
Let $\rho^{(N)}_n$ denote the  $n$-point correlation
function of the measure $\mu =\mu_{\beta, V}^{(N)}$  defined by
\begin{equation}\label{eqn:corrFunct}
\rho^{(N)}_n(x_1,\ldots,x_n)=
\int_{\RR^{N-n}}\tilde\mu(x)\rd x_{n+1}\dots\rd x_{N},
\end{equation}
where $\tilde \mu$ is the  symmetrized version of $\mu$ given in \eqref{eqn:measure}
but defined  on  $\RR^N$ instead of the simplex:
$$
\tilde\mu^{(N)}(\rd\lambda)=\frac{1}{N!}\mu(\rd\lambda^{(\sigma)}),
$$
where
$\lambda^{(\sigma)}=(\lambda_{\sigma(1)},\dots,\lambda_{\sigma(N)})$, with
$\lambda_{\sigma(1)}<\dots<\lambda_{\sigma(N)}$.

{F}rom   \eqref{eq:Mainm}
we have the convergence of the correlation functions, stated  as the following corollary.
Since the proof is a standard argument and it
 is essentially identical to  the one given  in Section 7 of  \cite{ESYY},
we omit it.

\begin{corollary}
Under the assumption of Theorem \ref{thm:Main}  and with the same notations,
for any smooth test functions $O$ with compact support  and for any $0<k\le \frac{1}{2}$,
we have,  with $s:=N^{-1+k}$, that
\begin{align*}
\lim_{N \to \infty} \int  & \rd \alpha_1 \cdots \rd \alpha_n\, O(\alpha_1,
\dots, \alpha_n) \Bigg [
  \int_{E - s}^{E + s} \frac{\rd x}{2 s}  \frac{1}{ \varrho (E)^n  }  \rho_n^{(N)}   \Big  ( x +
\frac{\alpha_1}{N\varrho(E)}, \dots,   x + \frac{\alpha_n}{N\varrho(E)}  \Big  ) \\
&
-   \int_{E' - s}^{E' + s} \frac{\rd x}{2 s}  \frac{1}{\varrho_{sc}(E)^n}
\rho_{{\rm Gauss}, n}^{(N)}
   \Big  ( x +
\frac{\alpha_1}{N\varrho_{sc}(E)}, \dots,   x + \frac{\alpha_n}{N\varrho_{sc}(E)}  \Big  ) \Bigg ]
\;=\; 0\,.
\end{align*}
\end{corollary}

The local statistics of the $\lambda_i's$ in the Gaussian case have been explicitly computed by Gaudin, Mehta and Dyson
(see, e.g., \cite{M})
for the classical value $\beta \in \{1, 2, 4\}$. For general $\beta> 0$, there is
an explicit description in terms of some stochastic differential equations, the {\it Brownian carousel}  \cite{VV}.

Theorem \ref{thm:Main} will be proved in two steps as explained in the introduction.
For logical reasons, we will first present  Step 2 on  particle location estimates in
Section 3 and then
Step 1 on the uniqueness of Gibbs measure in a finite interval in Sections 4 and 5.

\section{Optimal accuracy for particle locations}

Along this section, we assume that $V$ satisfies the same conditions as in Theorem \ref{thm:Main}.
Let the typical position  $\gamma_k$ be
defined by
$$
 \int_{-\infty}^{\gamma_k}\rho(s)\rd s=\frac{k}{N}.
$$
Moreover, all constants in this section depend on the potential $V$, which is fixed.
In the following, we will denote $\llbracket x,y\rrbracket=\NN\cap[x,y]$,

The purpose of this section is to prove that accuracy holds
for the measure $\mu$ at the optimal scale $1/N$, in the following sense.

\begin{theorem}\label{thm:accuracy}
Take any $\alpha>0$ and $\epsilon>0$. There are constants
$\delta,c_1,c_2>0$ such that for any $N\geq 1$ and $k\in\llbracket \alpha N,(1-\alpha) N\rrbracket$,
$$
\P_\mu\left(|\lambda_k-\gamma_k|> N^{-1+\epsilon}\right)\leq c_1e^{-c_2N^\delta}.
$$
\end{theorem}

After some initial estimates relying on large deviations results, the proof consists in comparing
$\mu$ to some {\it locally constrained measures} for which better concentration estimates can be proved  for the differences between particles.  This measure is related to the pseudo-equilibrium measure in \cite{ESY4}, but has distinctly different properties.
Iterations of these comparisons will give optimal accuracy.

\subsection{Initial estimates}

The purpose of this paragraph it to prove the following crude estimate. It will be the initial step in the induction of Subsection \ref{subsec:induction}.

\begin{proposition}\label{prop:InitialEstimate}
For any $\alpha, \epsilon>0$
there are constants $c_1,c_2,\delta>0$ such that for any $N$ and $k\in\llbracket \alpha N,(1-\alpha)N\rrbracket$
\begin{equation}\label{eqn:InitialEstimate}
\P_\mu\left(|\lambda_k-\gamma_k|>N^{-\frac{1}{2}+\epsilon}\right)\leq c_1e^{- c_2 N^\delta}.
\end{equation}
\end{proposition}

This result is a direct consequence of the following equation
(\ref{eqn:initialconcentration}) and Corollary \ref{cor:InitialAccuracy}, whose proofs are the purpose of this section. We
first state well-known facts about the equilibrium measure.

For convex  analytic potential  $V$ satisfying the asymptotic growth condition (\ref{eqn:LSImu}) (or even with weaker hypotheses on $V$, see e.g. \cite{BPS, APS}),   the equilibrium measure $\rho(s)\rd s$ associated with  $(\mu^{(N)})_{N\geq 0}$
can be defined as the unique minimizer (in the set of probability measures on $\RR$ endowed with the weak topology) of the
functional
$$
I(\nu)=
\int V(t)\rd\nu(t)-
\iint\log|t-s|\rd\nu(s)\rd\nu(t)
$$
if $\int V(t)\rd\nu(t)<\infty$, and $I(\nu)=\infty$ otherwise.
Moreover, $\rho$
has the following properties:
\begin{enumerate}[(a)]
\item The support of $\rho$ is a single interval $[A,B]$.
\item This equilibrium measure satisfies
\be
   \frac{1}{2}V'(t) = \int \frac{\rho(s)\rd s}{t-s}.
\label{equilibrium}
\ee
for any $t\in(A,B)$.
\item For any $t\in[A,B]$,
\begin{equation}\label{eqn:rho}
\rho(t)\ind t=\frac{1}{\pi}r(t)\sqrt{(t-A)(B-t)}\mathds{1}_{[A,B]}\ind t,
\end{equation}
where $r$ can be extended into an analytic function in $\CC$ satisfying
\begin{equation}\label{eqn:r}
r(z)=\frac{1}{2\pi}\int_A^B\frac{V'(z)-V'(t)}{z-t}\frac{\rd t}{\sqrt{(t-A)(B-t)}}.
\end{equation}
In particular, for convex $V$, $r$ has no zero in $\RR$.
\end{enumerate}

It is known that the particles locations cannot be far from its classical location  \cite{BenGui, Sch}: for any
 $\epsilon>0$ there are positive constants $C$, $c$, such that, for all $N\geq 1$,
\begin{align}
\P_\mu\left( \exists k\in\llbracket1,N\rrbracket\mid
 | \lambda_k-  \gamma_k| \ge \epsilon \right)\leq C e^{-c N^c  }. \label{eqn:largDev1}
\end{align}
In order to have density strictly in a compact support,
for given $R>0$,
define the following variant of $\mu^{(N)}$ conditioned to have all particles in $[-R,R]$:
\begin{equation}\label{eqn:truncMeasure}
\mu^{(N,R)}(\rd\lambda)=\frac{1}{Z_{N,R}}
\prod_{1\leq i<j\leq N}|\lambda_i-\lambda_j|^\beta\prod_{k=1}^N e^{-N\frac{\beta}{2}V(\lambda_k)}\mathds{1}_{|\lambda_k|<R}
\rd\lambda_1\dots\rd\lambda_N.
\end{equation}
Let  $\rho_k^{(N,R)}$ denote the marginals of the measure $\mu^{(N,R)}$, i.e. the
same definition as (\ref{eqn:corrFunct}), but with $\mu^{(N)}$ replaced by $\mu^{(N,R)}$.

Then Lemma 1 in \cite{BPS} states that
under condition (\ref{eqn:LSImu}) there exist some $R>0$ and  $c>0$, depending only on $V$,  such that for any $|x_1|,\dots,|x_k|\leq R$
\begin{equation}\label{eqn:BPS1}
\left|\rho^{(N,R)}_k(x_1,\dots,x_k)-\rho^{(N)}_k(x_1,\dots,x_k)\right|\leq \rho^{(N,R)}_k(x_1,\dots,x_k)e^{-c N},
\end{equation}
and for $|x_1|,\dots,|x_j|\geq R$, $|x_{j+1}|,\dots,|x_k|\leq R$,
\begin{equation}\label{eqn:BPS2}
\rho^{(N)}_k(x_1,\dots,x_k)\leq e^{-c N\sum_{i=1}^k\log |x_i|}.
\end{equation}
The last type of estimates we need are concentration and accuracy of the particles location at scale $N^{-1/2}$, in the bulk.
Concentration is a simple consequence of the Bakry-\'Emery convexity criterion for the logarithmic
Sobolev inequality (\cite{BakEme}, see also \cite{AGZ}):
define $\mathcal{H}$ by $\mu(\rd\lambda)=\frac{1}{Z_N}e^{- N \mathcal{H}(\lambda)}\rd \lambda$, and assume
\begin{equation}\label{eqn:BakEme}
\nabla^2 \mathcal{H}\geq  \sigma  \,\Id_N
\end{equation}
in the sense of partial order for positive
definite operators. Then $\mu$ satisfies a logarithmic Sobolev inequality with constant $2/(\sigma N)$:
for any probability density $f$ we have
\be\label{eqn:lsi}
\E_\mu f \log f  \le  \frac 2 { \sigma  N} \E_\mu |\nabla \sqrt f |^2 .
\ee
It is well-known that the logarithmic Sobolev inequality implies the spectral gap and,
together with Herbst's lemma, it also implies that
for any $k\in\llbracket 1,N\rrbracket$ and $x>0$
$$
\P_\mu\left(|\lambda_k-\E_\mu(\lambda_k)|>x\right)\leq 2e^{-\sigma N x^2/2 }.
$$
In our case where $\mu$ is defined by (\ref{eqn:measure}), for any $v\in\RR^N$
\be
v^*(\nabla^2 H)v=\frac { \beta} N \sum_{i<j}\frac{(v_i-v_j)^2}{(\lambda_i-\lambda_j)^2}+\frac{\beta}{2}\sum_iV''(\lambda_i)v_i^2\geq \varpi
 |v|^2
\label{eqn:Hconv},
\ee
where $\varpi$ is defined in (\ref{eqn:LSImu}).
Thus there is  a
constant $\tilde c>0$ such that for any $k$
\begin{equation}\label{eqn:initialconcentration}
\P_\mu\left(|\lambda_k-\E_\mu(\lambda_k)|>x\right)\leq 2e^{-\tilde c Nx^2},
\end{equation}
i.e. concentration at scale $1/\sqrt{N}$ holds. We now prove that accuracy at the same scale
holds inside the bulk.

The proof of the following lemma  is based on  an argument in \cite{Joh} for the polynomial case.
In the form presented here, it follows very closely the proof in \cite{Sch} for
the analytic case except that we use the logarithmic Sobolev inequality to have a more precise estimate.
We now introduce some notations needed in the proof.
\begin{itemize}
\item $m_N$ is the Stieljes transform of $\rho^{(N)}_1(s)(\rd s)$, evaluated at some $z$ with $\Im(z)>0$, and $m$ its
limit:
$$
m_N(z)=\E_\mu\left(\frac{1}{N}\sum_{k=1}^N\frac{1}{z-\lambda_i}\right)=\int_\RR\frac{1}{z-t}\rho_1^{(N)}(t)\rd t,\
m(z)=\int_\RR\frac{1}{z-t}\rho(t)\rd t.
$$
It is well-known that uniformly in any $\{\Im(z)>\epsilon\}$, $\epsilon>0$,
 $|m_N-m|\to 0$ (see e.g. \cite{AGZ}).
Along the proof of the next Lemma \ref{lem:Johansson}
 we will see that this convergence holds at speed $1/N$.

\item $s(z)=-2r(z)\sqrt{(A-z)(B-z)}$, where the square root is defined such that
$$
f(z)=\sqrt{(A-z)(B-z)}\sim z \quad \text{ as }  \quad z\to\infty;
$$
\item $b_N(z)$ is an analytic function defined by
$$
b_N(z)=\int_{\RR}\frac{V'(z)-V'(t)}{z-t}(\rho_1^{(N)}-\rho)(t)\ind t;
$$
\item finally, $c_N(z)=\frac{1}{N^2}k_N(z)+\frac{1}{N}\left(\frac{2}{\beta}-1\right)m_N'(z)$, where
$$
k_N(z)=\var_\mu\left(\sum_{k=1}^N\frac{1}{z-\lambda_i}\right).
$$
Here the $\var$ of a complex random variable
denotes $\var(X)=\E(X^2)-\E(X)^2$, i.e. without absolute value
unlike for the usual variance.
 Note
that $|\var(X)|\leq \E(|X-\E(X)|^2)$.
\end{itemize}
 The equation used by Johansson (which can be obtained by a change of variables
in (\ref{eqn:measure}) \cite{Joh} or by integration by parts \cite{Sch}), is a variation of the loop equation (see, e.g., \cite{Ey})
used in physics literatures and it takes the form
\begin{equation}\label{eqn:firstLoop}
(m_N-m)^2+s(m_N-m)+b_N=c_N.
\end{equation}

Equation  \eqref{eqn:firstLoop} expresses the difference $m_N - m$ in terms of $(m_N-m)^2$, $b_N$ and $c_N$.
In the  regime where $|m_N - m|$  is small, 
we can neglect the quadratic term. The term $b_N$ is the same order as
$|m_N-m|$ and is  difficult  to treat. As observed in \cite{APS,Sch}, for analytic  $V$,  this term vanishes
when we perform a contour integration. So we have roughly the relation
\be\label{55}
(m_N-m) \sim \frac 1 { N^2} \var_\mu\left(\sum_{k=1}^N\frac{1}{z-\lambda_k}\right),
\ee
where we dropped the less important error involving $m_N'(z)/N $ due to the extra $1/N$ factor.
 In the convex setting,
the variance can be estimated by  the logarithmic Sobolev
inequality and we immediately obtain an estimate on $m_N - m$.
 We then  follow the method in \cite {ErdRamSchYau}
to use the Helffer-Sj\"ostrand functional calculus to have an estimate on the particle  locations.
 Although it is tempting
to use  this new accuracy information on the particle locations to estimate the variance again in \eqref{55},
 this naive bootstrap
is difficult to implement. The main reason is, roughly speaking,  that
the particle location  estimate obtained from  knowing only the size of $m_N-m$
is not strong enough in the bootstrap.
The key idea in this section is the observation that {\it accuracy information on particle locations can be used to  improve the local convexity
of the measure $\mu$ in  the direction involving the  differences of particle locations}, see Lemma \ref{lem:localConvexity}.
Now we are able to complete the bootstrap argument and obtain a more accurate estimate on $m_N-m$.  Since this argument can be repeated, we can estimate the locations of particles  up to the optimal scale in the bulk.

\begin{lemma}\label{lem:Johansson}
Let $\delta>0$.
For $z=E+\ii \eta$ with $A+\delta<E<B-\delta$
assume that
\begin{equation}\label{eqn:kNTo0}
\frac{1}{N^2}k_N(z)\to 0
\end{equation}
 as $N\to\infty$ uniformly in $\eta\geq N^{-1+a}$ for some $0<a<1$. Then there
 are  constants $c,\e>0$  such that for any
$N^{-1+a}\le\eta\le\e$,
$A+\delta<E<B-\delta$,
\begin{equation}\label{eqn:lemJohansson}
|m_N(z)-m(z)|\leq c\left(\frac{1}{N\eta}+\frac{1}{N^2}k_N(z)\right).
\end{equation}
\end{lemma}

\begin{proof}
First, for technical contour integration reasons, it will be easier to consider the measure (\ref{eqn:truncMeasure})
instead of $\mu^{(N)}$ here. More precisely, define
\begin{align*}
m^{(R)}_N(z)&=
\E_{\mu^{(N,R)}}\left(\frac{1}{N}\sum_{k=1}^N\frac{1}{z-\lambda_i}\right)=
\int_\RR\frac{1}{z-t}\rho_1^{(N,R)}(t)\rd t,\\
k_{N}^{(R)}(z)&=\var_{\mu^{(N,R)}}\left(\sum_{k=1}^N\frac{1}{z-\lambda_i}\right),\\
c_N^{(R)}(z)&=\frac{1}{N^2}k_N^{(R)}(z)+\frac{1}{N}\left(\frac{2}{\beta}-1\right){m_N^{(R)}}'(z).
\end{align*}
Then it is a direct consequence of (\ref{eqn:BPS1}) and (\ref{eqn:BPS2}) that there are constants $c>0$ and
$R>0$ such that
uniformly on
$\eta\geq N^{-10}$ (or any power of $N$),
\begin{equation}\label{eqn:expDiff}
|m^{(R)}_N-m_N|=\OO\left(e^{-c N}\right),\ \ \ |k_{N}^{(R)}-k_{N}|=\OO(e^{-c N}).
\end{equation}
Consider the rectangle with vertices
$
2R+\ii N^{-10},- 2R+\ii N^{-10}, -2R-\ii N^{-10}, 2R-\ii N^{-10}
$,
call $\mathcal{L}$ the corresponding clockwise closed contour and  $\mathcal{L}'$
the one consisting only in the horizontal pieces, with the same orientation.
{F}rom $(\ref{eqn:firstLoop})$, we obviously have, for $z\not\in\mathcal{L}'$,
$$
\frac{1}{ 2\pi\ii}\int_{\mathcal{L}'}\frac{(m_N(\xi)-m(\xi))^2+s(\xi)(m_N(\xi)-m(\xi))
+b_N(\xi)-c_N(\xi)}
{r(\xi)(z-\xi)}\rd\xi=0.
$$
Note that the above expression makes sense for large enough $N$, because then $r$ has no zero on $\mathcal{L}$.
Using (\ref{eqn:expDiff}), this implies, for $\eta\geq N^{-1}$,
$$
\frac{1}{2\pi\ii}\int_{\mathcal{L}'}\frac{(m^{(R)}_N(\xi)-m(\xi))^2
+s(\xi)(m^{(R)}_N(\xi)-m(\xi))
+b_N(\xi)-c^{(R)}_N(\xi)}{r(\xi)(z-\xi)}\rd\xi=
\OO(e^{-c N}).
$$
Now, as $\rho_1^{(N,R)}$ and $\rho$ are supported on $[-R,R]$, $m_N^{(R)}-m$
and $c_N^{(R)}$ are uniformly $\OO(1)$ in the vertical segments of $\mathcal{L}$. Consequently, from the above equation
$$
\frac{1}{2\pi\ii}\int_{\mathcal{L}}\frac{(m^{(R)}_N(\xi)-m(\xi))^2+s(\xi)(m^{(R)}_N(\xi)-m(\xi))+b_N(\xi)-c^{(R)}_N(\xi)}{r(\xi)(z-\xi)}\rd\xi=
\OO(N^{-10}).
$$
As $b_N$ and $r$ are analytic inside $\mathcal{L}$, for $z$  outside
 $\mathcal{L}$ we get
$$
\frac{1}{2\pi\ii}\int_{\mathcal{L}}\frac{(m^{(R)}_N(\xi)-m(\xi))^2+s(\xi)(m^{(R)}_N(\xi)-m(\xi))-c^{(R)}_N(\xi)}{r(\xi)(z-\xi)}\rd\xi=
\OO(N^{-10}).
$$
Remember we define  $f(z)=\sqrt{(A-z)(B-z)}$ uniquely by $f(z)\sim z$ as $z\to\infty$. Moreover, $|m_N^{(R)}-m|(z)=\OO(z^{-2})$ as $|z|\to\infty$ because
$\rho$ and $\rho_1^{(N,R)}$ are compactly supported:
\begin{multline*}
|m_N^{(R)}(z)-m(z)|=\left|\int_{-R}^R\frac{\rho(t)-\rho^{(N,R)}(t)}{z-t}\rd t\right|\\=\left|\int_{_R}^R(\rho(t)-\rho^{(N,R)}(t))\left(\frac{1}{z}+\OO\left(\frac{1}{z^2}\right)\right)\rd t\right|=\OO\left(z^{-2}\right).
\end{multline*}
Consequently, the function $s(m^{(R)}_N-m)/r=-2f(m^{(R)}_N-m)$ is $\OO(z^{-1})$ as $|z|\to\infty$. Moreover, it is analytic outside $\mathcal{L}$, so
the Cauchy integral formula yields
$$
\frac{1}{2\pi\ii}\int_{\mathcal{L}}\frac{s(\xi)(m^{(R)}_N(\xi)-m(\xi))}{r(\xi)(z-\xi)}\rd\xi=-2f(z)(m_N^{(R)}-m)(z),
$$
proving
\begin{equation}\label{eqn:withoutB}
-2f(z)(m_N^{(R)}(z)-m(z))=
-\frac{1}{2\pi\ii}\int_{\mathcal{L}}
\frac{(m^{(R)}_N(\xi)-m(\xi))^2-c^{(R)}_N(\xi)}{r(\xi)(z-\xi)}\rd\xi+
\OO(N^{-10}).
\end{equation}
Consider now the following rectangular contours, defined by their vertices:
\begin{align}
{\mathcal{L}}_1:\ &\notag
R+\epsilon+\ii \epsilon,- R-\epsilon+\ii \epsilon, -R-\epsilon-\ii \epsilon, R+\epsilon-\ii \epsilon,\\
{\mathcal{L}}_2:\ &\label{eqn:contour}
R+2\epsilon+2\ii \epsilon,- R-2\epsilon+2\ii \epsilon, -R-2\epsilon-2\ii \epsilon, R+2\epsilon-2\ii \epsilon,
\end{align}
where $\epsilon>0$ is fixed, small enough such that all zeros of $r$ are strictly outside $\mathcal{L}_2$.
For $z$ inside $\mathcal{L}_2$ and $\Im(z)\geq N^{-1}$, by the Cauchy formula, equation (\ref{eqn:withoutB}) implies that
\begin{multline}\label{eqn:onL1}
-2s(z)(m_N^{(R)}(z)-m(z))\\=-(m^{(R)}_N(z)-m(z))^2+c^{(R)}_N(z)
-\frac{1}{2\pi\ii}\int_{\mathcal{L}_2}\frac{(m^{(R)}_N(\xi)-m(\xi))^2-c^{(R)}_N(\xi)}
{r(\xi)(z-\xi)}\rd\xi+
\OO(N^{-10}).
\end{multline}
In the above expression, if now $z$ is on $\mathcal{L}_1$,
$|z-\xi|>\epsilon$, and on $\mathcal{L}_2$
$|r|$
is separated away
from zero by a positive universal constant. Moreover,
$c_N^{(R)}(\xi)$ can be bounded in the following way: for any constants $\alpha_1,\dots,\alpha_N\in [-R-\epsilon,R+\epsilon]$,
\begin{multline*}
\frac{1}{N^2}\left|\var_{\mu^{(N,R)}}\left(\sum_{k=1}^N\frac{1}{\xi-\lambda_k}\right)\right|
\leq
\frac{1}{N^2}
\E_{\mu^{(N,R)}}\left(\left|\sum_{k=1}^N
\frac{1}{\xi-\lambda_k}-\frac{1}{\xi-\alpha_k}\right|^2\right)\\\leq \frac{1}{\epsilon^4N}
\sum_{k=1}^N\E_{\mu^{(N,R)}}\left(|\lambda_k-\alpha_k|^2\right),
\end{multline*}
because for any $k$, we have $|\xi-\lambda_k|>\epsilon$, $|\xi-\alpha_k|>\epsilon$.
Now, choose $\alpha_k=\E_{\mu}(\lambda_k)$. By (\ref{eqn:BPS2}),
for large enough $N$ any $\alpha_k$, $1\leq k\leq N$,
is in $[-R-\epsilon,R+\epsilon]$ indeed. Moreover, by (\ref{eqn:BPS1}),
$$
\left|\E_{\mu^{(N,R)}}\left(|\lambda_k-\alpha_k|^2\right)
-\E_{\mu}\left(|\lambda_k-\alpha_k|^2\right)
\right|
\leq e^{-cN}\E_{\mu^{(N,R)}}\left(|\lambda_k-\alpha_k|^2\right),
$$
and, by the spectral gap inequality for $\mu$,
$\E_{\mu}\left(|\lambda_k-\alpha_k|^2\right)=\OO(N^{-1})$.
This together proves that $k^{(R)}_N(\xi)$ is $\OO(N^{-1})$, uniformly on the contour $\mathcal{L}_2$.
Moreover,
$\frac{1}{N}{m^{(R)}_N}'=\OO(N^{-1})$, so finally $c_N^{(R)}(\xi)$ is
uniformly $\OO(N^{-1})$ on $\mathcal{L}_2$ and  (\ref{eqn:onL1}) implies
$$
-2s(z)(m_N^{(R)}(z)-m(z))=-(m^{(R)}_N(z)-m(z))^2(z)+\OO\left(\sup_{\mathcal{R}_2}|m_N^{(R)}-m|^2\right)+
\OO(N^{-1}).
$$
Moreover, from the maximum principle for analytic functions,
$\sup_{\mathcal{L}_2}|m_N^{(R)}-m|
\leq
\sup_{\mathcal{L}_1}|m_N^{(R)}-m|$, so the previous equation implies
$$
\sup_{\mathcal{L}_1}|m_N^{(R)}-m|=\OO\left(\sup_{\mathcal{L}_1}|m_N^{(R)}-m|^2+\frac{1}{N}\right).
$$
We know that $\rho_1^{(N)}(s)\rd s$ converges weakly to $\rho (s)\rd s$ (see \cite{AGZ}),
so by (\ref{eqn:BPS1}) and (\ref{eqn:BPS2}) $\rho_1^{(N,R)}(s)\rd s$ converges weakly to  $\rho (s)\rd s$.
On $\mathcal{L}_1$, $z$ is at distance at least $\epsilon$ from the support of both $\rho_1^{(N,R)}(s)\rd s$
and $\rho(s)\rd s$ so, on $\mathcal{L}_1$, $m_N^{(R)}-m$ converges
 uniformly to $0$. Together with the above equation, this implies that
$$
\sup_{\mathcal{L}_1}|m_N^{(R)}-m|=\OO\left(\frac{1}{N}\right).
$$
By the maximum principle the same estimate holds outside $\mathcal{L}_1$,
in particular on $\mathcal{L}_2$, so
equation (\ref{eqn:onL1}) implies that for $z$ inside $\mathcal{L}_1$
\begin{equation}\label{eqn:ReFirstLoop}
-2s(z)(m_N^{(R)}(z)-m(z))=-(m^{(R)}_N(z)-m(z))^2+c^{(R)}_N(z)+\OO\left(\frac{1}{N}\right).
\end{equation}
Moreover,
\begin{multline}\label{eqn:mNPrime}
\frac{1}{N}  | {m_N^{(R)}}'(z)| = \frac{1}{N^2 }  \left |
\E_{\mu^{(N,R)}}    \sum_j  \frac {1}  {(z-\lambda_j)^2} \right |\\
\le  \frac{1}{N \eta  } \Im   \, m^{(R)}_N(z)\leq \frac{1}{N \eta  } |m^{(R)}_N(z)-m(z)|+\frac{1}{N\eta}|\Im\ m(z)|\leq
\frac{1}{N \eta  } |m^{(R)}_N(z)-m(z)|+\frac{c}{N\eta}
\end{multline}
for some constant $c$.
We used the well-known fact that $\Im\ m$ is uniformly bounded on the upper half plane\footnote{This follows for example from properties of the Cauchy operator, see p 183 in \cite{De1}.}.
On the set $A+\delta <E<B-\delta$ and $|\eta|<\epsilon$, we have  $\inf |s|>0$. Therefore
(\ref{eqn:ReFirstLoop}) takes the form
\begin{equation}\label{eqn:upBoundSt}
\left(1+\OO\left(\frac{1}{N\eta}\right)\right)(m^{(R)}_N(z)-m(z))=
\OO\left(|m^{(R)}_N(z)-m(z)|^2+\frac{1}{N^2}k^{(R)}_N(z)+\frac{1}{N\eta}\right).
\end{equation}
{F}rom the hypothesis (\ref{eqn:kNTo0}), if $N^{-1+a}\leq \eta\leq\epsilon$
and $A+\delta <E<B-\delta$, then
\begin{equation}\label{eqn:quadratic}
|m^{(R)}_N-m|\leq c|m^{(R)}_N-m|^2+\epsilon_N,
\end{equation}
for some $c>0$ and $\epsilon_N\to 0$ as $N\to\infty$.
For large $N$, (\ref{eqn:quadratic}) implies that
$|m^{(R)}_N-m|\leq 2\epsilon_N$ or $|m^{(R)}_N-m|\geq 1/c-2\epsilon_N$.
 Together with $|m^{(R)}_N-m|(E+\epsilon \ii)\to 0$
and the continuity of
$|m^{(R)}_N-m|$ in the upper half plane, this implies that
$|m^{(R)}_N-m|\leq 2\epsilon_N$ and therefore
$|m^{(R)}_N-m|\to 0$ uniformly on $N^{-1+a}\leq \eta\leq\epsilon$, $A+\delta <E<B-\delta$.
Consequently, using  (\ref{eqn:upBoundSt}), this proves that
there is a constant $c>0$  such that for any
$\eta\geq N^{-1+a}$,
$A+\delta<E<B-\delta$,
$$
|m^{(R)}_N(z)-m(z)|\leq c\left(\frac{1}{N\eta}+\frac{1}{N^2}k^{(R)}_N(z)\right).
$$
The same conclusion remains when substituting $m_N^{(R)}$ (resp. $k_N^{(R)}$)
by $m_N$ (resp. $k_N$) thanks to (\ref{eqn:BPS1}) and (\ref{eqn:BPS2}).

\end{proof}
\vspace{0.3cm}

To prove accuracy results for $\mu$, the above Lemma \ref{lem:Johansson} will be combined with the following one.

\begin{lemma}\label{lem:HS} Let $\delta<(B-A)/2$ and $E\in[A+\delta,B-\delta]$ and $0<\eta<\delta/2$.
Define a function $f= f_{E,\eta}$: $\R\to \R$
 such that $f(x) = 1$ for $x\in (-\infty, E-\eta]$, $f(x)$ vanishes
for  $x\in [E+\eta, \infty)$, moreover
 $|f'(x)|\leq c\eta^{-1}$ and $|f''(x)|\leq c\eta^{-2}$, for some constant $c$.
 Let $\wt\rho$ be an arbitrary signed measure
and let $S(z)= \int (z-x)^{-1}\wt\rho(x)\rd x$ be its Stieltjes transform.
Assume that, for any $x\in[A+\delta/2,B-\delta/2]$,
\begin{equation}\label{eqn:cond1}
\left| S(x+iy)\right|\leq \frac{ U}{Ny}\;\;  \mbox{for}\;\; \eta <y<1 ,\;\;\mbox{and}\;\;
|\Im\, S(x+iy)|\leq \frac{ U}{Ny}\;\; \mbox{for}\;\; 0<y<\eta.
\end{equation}
 Assume moreover that $\int_\RR\wt\rho(\lambda)\rd\lambda=0$ and that there is a real 
constant $\mathcal{T}$ such that 
\begin{equation}\label{eqn:cond2}
\int_{[-\mathcal{T},\mathcal{T}]^c}  |\lambda\wt\rho(\lambda)|\rd\lambda   \le
\frac{U}{N}.
\end{equation}
Then for some  constant $C>0$, independent of $N$ and $E\in [A+\delta,B-\delta]$, we have
$$
\left|\int f_E(\lambda)\wt\rho(\lambda)\rd\lambda \right|  \le
\frac{C U|\log\eta| }{N}.
$$
\end{lemma}

\begin{proof}
Our starting point, relying on the
Helffer-Sj\"ostrand functional calculus, is formula (B.13) in \cite{ErdRamSchYau}:
\begin{align}\label{eqn:HSbound1}
\left|\int_{-\infty}^\infty f_E(\lambda)\wt\rho(\lambda)\rd\lambda\right|
\leq&
C\left|\iint y f_E''(x)\chi(y)\Im S(x+\ii y)\rd x\rd y\right|\\
&+\label{eqn:HSbound2}
C\iint\left(|f_E(x)|+|y||f_E'(x)|\right)|\chi'(y)|\left|S(x+\ii y)\right|\rd x\rd y,
\end{align}
for some universal $C>0$, and
where $\chi$ is a smooth cutoff function with support in $[-1, 1]$, with $\chi(y)=1$ for $|y| \leq 1/2$ and with bounded derivatives.

Using (\ref{eqn:cond1}) and (\ref{eqn:cond2}), the support of $\chi'$ being
included in $1/2\leq|y|\leq 1$, and the fact that $f_E'$ is $\OO(\eta^{-1})$
on an interval of size $\OO(\eta)$, the term (\ref{eqn:HSbound2})
is easily bounded by $\OO\left(\frac{U}{N}\right)$.
Concerning the right hand side of (\ref{eqn:HSbound1}), following \cite{ErdRamSchYau} we split it depending on $0<y<\eta$ and $\eta<y<1$. Note that
by symmetry we only need to consider positive $y$.
The integral on the first integration regime is easily bounded by
$$
\left|\iint_{0<y<\eta}y f_E''(x)\chi(y)\Im S(x+\ii y)\rd x\rd y\right|=\OO\left(\iint_{|x-E|<\eta, 0<y<\eta}y\eta^{-2}\frac{U}{Ny}\rd x\rd y\right)=\OO\left(\frac{U}{N}\right).
$$
For the second integral, as $f_E''$ and $\chi$ are real, we can substitute $\Im m$ by $m$ and use the analyticity of $m$ when integrating by parts (first in $x$, then in $y$):
\begin{align*}
\left|\iint_{\eta<y}y f_E''(x)\chi(y)\Im S(x+\ii y)\rd x\rd y\right|
\leq&\left|\iint_{\eta<y}y f_E''(x)\chi(y)S(x+\ii y)\rd x\rd y\right|\\
=&\left|\iint_{\eta<y}y f_E'(x)\chi(y)S'(x+\ii y)\rd x\rd y\right|\\
\leq&
\left|\iint_{\eta<y}\partial_y(y\chi(y)) f_E'(x)S(x+\ii y)\rd x\rd y\right|\\
&+
\left|\int\eta f_E'(x)\chi(\eta)S(x+\ii \eta)\rd x\right|.
\end{align*}
This last integral is easily bounded by $\OO(U/N)$ using $(\ref{eqn:cond1})$. Concerning the previous one, as $f_E'=\OO(\eta^{-1})$, $|x-E|<\eta$ for non vanishing $f_E'$, $\partial_y(y\chi(y))=\OO(1)$ and $S(x+\ii y)=\OO(U/(Ny))$,
this is bounded by
$$\OO\left(\frac{U}{N}\int_{\eta}^1\frac{\rd y}{y}\right)=\OO\left(\frac{U|\log\eta|}{N}\right),$$
concluding the proof.
\end{proof}

As a corollary of Lemmas \ref{lem:Johansson} and \ref{lem:HS}, we get the accuracy at scale $1/\sqrt{N}$ for
the $\lambda_k$'s in the bulk.

\begin{corollary}\label{cor:InitialAccuracy}
For any $\alpha>0$ and $\epsilon>0$ we have
$$
|\gamma^{(N)}_k-\gamma_k|=\OO\left(N^{-1/2+\epsilon}\right)
$$
uniformly in $k\in\llbracket \alpha N,(1-\alpha)N\rrbracket$ where
$\gamma_k^{(N)}$ and $\gamma_k$ are defined by
$$
\int_{-\infty}^{\gamma^{(N)}_k}
\rho_1^{(N)}(s)\rd s=\frac{k}{N},\;\; \mbox{and}  \;\; \int_{-\infty}^{\gamma_k}
\rho(s)\rd s=\frac{k}{N}.
$$
\end{corollary}

\begin{proof}
We will apply Lemma \ref{lem:HS} to $\wt\rho=\rho-\rho_1^{(N)}$ with  $\eta=N^{-1/2+\epsilon}$,
and check the conditions on $S=m-m_N$.
We denote $z=x+\ii y$.

By the spectral gap inequality for the measure $\mu$, we get
\begin{equation}\label{eqn:vg}
\frac{1}{N^2}\left|\var_\mu\left(\sum_{k=1}^N\frac{1}{z-\lambda_k}\right)\right|
\leq
\frac{c}{N^3}\E_\mu\left(\left|\nabla\left(\sum_{k=1}^N
\frac{1}{z-\lambda_k}\right)\right|^2\right)\leq \frac{c}{N^2 y^4}.
\end{equation}
Together with Lemma \ref{lem:Johansson}, this implies that uniformly in
$N^{-1/2+\epsilon}\leq y\leq 1$ and $x\in[A+\delta/2,B-\delta/2]$, we have
$$
|m_N(z)-m(z)|=\OO\left(\frac{1}{N}+\frac{1}{N^2 y^4}\right)=\OO\left(\frac{\sqrt{N}}{N y}\right).
$$
For $0< y<N^{-1/2+\epsilon}$, $m$ is uniformly bounded and
$$
y\mapsto y\,\Im m_N(x+\ii y)=\int\frac{y^2}{(x-t)^2+y^2}\rho^{(N)}(t)\rd t
$$
is an increasing function, so denoting $y_0=N^{-1/2+\epsilon}$ we have
$$
y\,|\Im(m_N(x+\ii y)-m(x+\ii y))|
\leq
y_0\,\Im m_N(x+\ii y_0)+\OO(y)
\leq y_0|\Im(m_N(x+\ii y_0)-m(x+\ii y_0))|+\OO(y_0)
.$$
Therefore, for any $0<y<N^{-1/2+\epsilon}$,
$$
|\Im(m_N(x+\ii y)-m(x+\ii y))|=\OO\left(\frac{N^{1/2+\epsilon}}{Ny}\right).
$$
 Finally, the  condition (\ref{eqn:cond2}) with $U=\OO(N^{1/2+\epsilon})$
 and with the choice of any  $\mathcal{T} \ge \max (|A|, |B|)+\delta$ follows from the large deviation
estimate \eqref{eqn:largDev1}.

Using  the conclusion of Lemma \ref{lem:HS}
for functions $f_E$ and $f_{E+\eta}$ defined in the same lemma,
and subtracting both results,
we get that uniformly in  $E\in[A+\delta,B-\delta]$,
\begin{equation}\label{eqn:convRho}
\left|\int_{-\infty}^E(\rho^{(N)}(t)-\rho(t))\rd t\right|=\OO(N^{-1/2+\epsilon}),
\end{equation}
so if $\gamma_k^{(N)}\in[A+\delta,B-\delta]$, then
$|\gamma^{(N)}_k-\gamma_k|=\OO\left(N^{-1/2+\epsilon}\right)$. This estimate holds
uniformly in $k\in\llbracket \alpha N,(1-\alpha)N\rrbracket$:
as a consequence of (\ref{eqn:convRho}) and the smooth form (\ref{eqn:rho}) of $\rho$, for any $k\in\llbracket \alpha N,(1-\alpha)N\rrbracket$ and sufficiently large $N$ we have $\gamma_k^{(N)}\in[A+\delta,B-\delta]$, for $\delta>0$ small enough, concluding the proof.
\end{proof}

\begin{lemma}\label{edge}
For any  $\epsilon>0$
there exists $c_1, c_2, \eps'$ positive
constants such that for any
$ N^{3/5+ \e}\le j \le N - N^{3/5+ \e} $, we have
$$
\P_\mu \left ( |\lambda_j-\gamma_j|\ge  N^{-4/15+\epsilon}\right)  \le c_1 e^{ - c_2  N^{\e'}}  \; .
$$
\end{lemma}

\begin{proof}
We
will assume that $j< N/2$ in the following, i.e. we will estimate the accuracy near the edge $A$,
the proof close to the other edge $B$ being analogous.
We follow the notations used in Corollary \ref{cor:InitialAccuracy} and Lemma \ref{lem:HS}.
For $E \in [A-\delta, A+\delta] \cup [B-\delta, B + \delta]$,   we have $\inf |f|(z) \ge \sqrt \eta $.
Therefore we can divide $-2 f(z)$ on both side of
(\ref{eqn:ReFirstLoop}) to have
\begin{equation}
\left(1+\OO\left(\frac{1}{N\eta ^{3/2}}\right)\right)(m^{(R)}_N(z)-m(z))=
\OO\left(   \frac {|m^{(R)}_N(z)-m(z)|^2} {\sqrt \eta}+
\frac{1}{N^2  \sqrt \eta}k^{(R)}_N(z)+\frac{1}{N\eta^{3/2}}\right).
\end{equation}
 By (\ref{eqn:vg}), (\ref{eqn:BPS1}) and (\ref{eqn:BPS2}) we can bound the variance term by
\be
\frac{1}{N^2  }k^{(R)}_N(z) \le \frac { c} { N^2 \eta^4}.
\ee
for $\eta\geq N^{-10}$ for example.
Following the same continuity argument in the proof of Lemma \ref{lem:Johansson}, we obtain that for any $\epsilon>0$
$$
|\Im(m_N-m)(x+\ii \eta)|=\OO\left(\frac{N^{\epsilon}}{N^2 \eta^{9/2} }\right),
$$
provided that
\be
 \frac {1} {\sqrt \eta}\left [ \frac{1}{N^2  \sqrt \eta}k^{(R)}_N(z)+\frac{1}{N\eta^{3/2}}  \right ] \le
 \frac { c} { N^2 \eta^5 } \ll 1.
\ee
We can now follow the argument in the proof of Corollary \ref{cor:InitialAccuracy} so that
\eqref{eqn:cond1} holds with $U = N^{3/5}$.  Since the condition \eqref{eqn:cond2} is easy to verify,
we thus  have
$$
\left|\int f_E(t ) \big [ \rho^{(N)}(t)-\rho(t)  \big ] \rd t  \right|  \le
\frac{C |\log\eta| }{N^{2/5}}, \quad \eta = N^{-2/5},
$$\
where $f_E$ is defined in Lemma \ref{lem:HS}.   This proves that
$$
\int^{E - \eta}_{-\infty}   \rho^{(N)}(t )   \rd t   \le   \int^{E + \eta}_{-\infty}   \rho(t)  \rd t +
\frac{C |\log\eta| }{N^{2/5}}.
$$
In particular,
$$
\frac j N -
\frac{C |\log\eta| }{N^{2/5}}
 = \int^{\gamma_j^{(N)} }_{-\infty}   \rho^{(N)}(t )   \rd t    -
\frac{C |\log\eta| }{N^{2/5}}   \le   \int^{\gamma_j^{(N)}  + 2 \eta}_{-\infty}   \rho(t)  \rd t,
$$
and we have, by definition of $\gamma_i$, that
$$
\gamma_{j - N^{3/5 + \e}} \le \gamma_j^{(N)}  + 2 \eta.
$$
Similarly, the reverse inequality holds and we have
$$
\gamma_{j - N^{3/5 + \e}}- 2 \eta  \le \gamma_j^{(N)}  \le \gamma_{j + N^{3/5 + \e}} + 2 \eta.
$$
Since $\int_A^E \rho(t) \rd t \sim (E-A)^{3/2}$,
for $j \ge N^{3/5 + \e}$ we have
$$
|\gamma_{j - N^{3/5 + \e}} - \gamma_j |\le C \left ( \frac j N \right )^{-1/3}
 N^{-2/5 + \e/2} \le N^{-4/15 + \e}.
$$
Moreover, by (\ref{eqn:initialconcentration}),
$\lambda_j$ is concentrated around $\E_\mu(\lambda_j)$
at scale $N^{-1/2}$, so
$|\E_\mu(\lambda_j)-\gamma^{(N)}_j|=\OO(N^{-1/2})$,
concluding the proof of the lemma.
\end{proof}

\subsection{The locally constrained measures}

In this section some arbitrary $\epsilon,\alpha>0$ are fixed.
Let $\theta$ be a continuous nonnegative function with $\theta=0$ on $[-1,1]$ and $\theta''\geq 1$ for $|x|>1$.
We can take for example $\theta(x)=(x-1)^2 \mathds{1}_{x>1}+(x+1)^2 \mathds{1}_{x<-1}$ in the following.

\begin{definition} \label{def:locallyConstrained}
For a given $k\in\llbracket\alpha N,(1-\alpha) N\rrbracket$ and any integer $1\leq  M\leq \alpha N$, we denote $I^{(k,M)}=\llbracket k-M,k+M\rrbracket$ and $i_M=|I^{(k,M)}|=2M+1$.
Moreover, let
$$
\phi^{(k,M)}=
\beta\sum_{i<j,i,j\in I^{(k,M)}}\theta\left(\frac{N^{1-\epsilon}(\lambda_i-\lambda_j)}
{i_M}\right).
$$
We define the probability measure
\begin{equation}\label{eqn:omega}
\rd\omega^{(k,M)}:=\frac{1}{Z}e^{-\phi^{(k,M)}}\rd\mu,
\end{equation}
where $Z=Z_{\omega^{(k,M)}}$.
The measure
$\omega^{(k,M)}$ will be referred to as locally constrained
transform of $\mu$, around $k$, with width $M$.
The dependence of
the measure on $\epsilon$ will be suppressed in the notation.
\end{definition}

We will also frequently use the following notation
for block averages in any sequence $x_1, x_2, \ldots $
\begin{equation}\label{eqn:average}
x_k^{[M]}:=\frac{1}{i_M}\sum_{i\in I^{(k,M)}}x_i.
\end{equation}

The reason for introducing these locally constrained measures is that they improve the convexity  in $I^{(k,M)}$
up to a common shift, as explained in the following lemma.

\begin{lemma}\label{lem:localConvexity}
Consider the previously defined
probability measure
$$\omega^{(k,M)}=\frac{1}{Z}e^{-\phi^{(k,M)}}\rd\mu=\frac{1}{\tilde Z}
e^{-N(\mathcal{H}_1+\mathcal{H}_2)}\rd\lambda,$$
where we denote
\begin{align*}
\mathcal{H}_1&:=\frac{1}{N}\phi^{(k,M)}-
\frac{\beta}{N}\sum_{i< j, i,j\in I^{(k,M)}}\log|\lambda_i-\lambda_j|
,\\
\mathcal{H}_2&:=-\frac{\beta}{N}\sum_{ (i, j)  \in J^{(k,M)}}\log|\lambda_i-\lambda_j|+\frac{\beta}{2}\sum_{i=1}^NV(\lambda_i)
\end{align*}
where $J^{(k,M)}$ is the set of pairs of points $i<j$ in $\llbracket 1,N \rrbracket$ such that $i$ or $j$ is not in $I^{(k,M)}$,
and $\mathcal{H}_1=\mathcal{H}_1(\lambda_{k-M},\dots,\lambda_{k+M})$.
Then $\nabla^2 \mathcal{H}_2\geq 0$
and  denoting $v=(v_i)_{i\in I^{(k,M)}}$, we also have
\begin{equation}\label{eqn:convexity}
v^*(\nabla^2 \mathcal{H}_1)v\geq
\frac{\beta}{2}\frac{N^{1-2\epsilon}}{i_M} \sum_{i,j\in I^{(k,M)}} (v_i-v_j)^2.
\end{equation}
\end{lemma}

\begin{proof}
Since $V$ is convex, 
to prove the convexity of $\mathcal{H}_2$, it suffices to prove it for the Coulomb interaction terms;
this relies on the calculation, for any $u\in\RR^N$,
$$
u^*(\nabla^2\mathcal{H}_2(\lambda))u=\frac{\beta}{N}\sum_{J^{(k,M)}}\frac{(u_i-u_j)^2}{(\lambda_i-\lambda_j)^2}\geq 0.
$$
Moreover, for any $v\in\RR^{i_M}$, a similar calculation yields
\begin{equation}\label{eqn:calculationConvexity}
v^*(\nabla^2 \mathcal{H}_1)v\geq\frac{\beta}{N} \sum_{i< j, i,j\in I^{(k,M)}}\frac{(v_i-v_j)^2}{(\lambda_i-\lambda_j)^2}
+
\beta\frac{N^{1-2\epsilon}}{i_M^2}\sum_{i< j, i,j\in I^{(k,M)}}(v_i-v_j)^2
\theta''\left(\frac{N^{1-\epsilon}(\lambda_i-\lambda_j)}{i_M}\right).
\end{equation}
{F}rom our definition of $\theta$,
$$
\frac{1}{(\lambda_i-\lambda_j)^2}+
\frac{N^{2-2\epsilon}}{i_M^2}\theta''\left(\frac{N^{1-\epsilon}(\lambda_i-\lambda_j)}{i_M}\right)
\geq
\frac{N^{2-2\epsilon}}{i_M^2},
$$
which implies
\begin{equation}\label{iInfj}
v^*(\nabla^2 \mathcal{H}_1)v\geq \beta \frac{N^{1-2\epsilon}}{i_M^2}\sum_{i< j, i,j\in I^{(k,M)}}(v_i-v_j)^2.
\end{equation}
which completes the proof of (\ref{eqn:convexity}),
noting that the above factor $1/2$ comes from the strict ordering of $i$ and $j$ indexes in (\ref{iInfj}).
\end{proof}

The above convexity bound, associated with the following local criterion for the logarithmic Sobolev inequality,
will yield a strong concentration for $\sum_{i\in I^{(k,M)}} v_i\lambda_i$ under $\omega^{(k,M)}$, if $\sum_i v_i=0$.
This lemma is a consequence of the Brascamp-Lieb
inequality \cite{BraLie}.  Notice that the original inequality applied only to measures on $\RR^N$,
but a mollifying argument in Lemma 4.4 of \cite{EKYY2}
has extended it to the measures on the simplex $\{ \lambda_1 < \lambda_2 < \ldots < \lambda_N\}$  considered in this paper.

\begin{lemma}\label{lem:localLSI} Decompose  the coordinates $\lambda=(\lambda_1, \ldots ,\lambda_N)$
of a point in  $\RR^N = \RR^m \times \RR^{N-m}$ as $\lambda=(x,y)$, where $x\in \RR^m$, $y\in \RR^{N-m}$.
Let $\omega=\frac{1}{Z}e^{-N \mathcal{H}}$ be a probability measure on $\RR^N = \RR^m \times \RR^{N-m}$ such that
$\mathcal{H}=\mathcal{H}_1+\mathcal{H}_2$, with $\mathcal{H}_1=\mathcal{H}_1(x)$ depending only on 
the $x$ variables and $\mathcal{H}_2 =\mathcal{H}_2(x,y)$
depending on all coordinates. Assume that, for any $\lambda\in\RR^N$,
$\nabla^2\mathcal{H}_2(\lambda)\geq 0$. 
Assume moreover that $\mathcal{H}_1(x)$ is independent of
 $x_1 + \ldots + x_m$, i.e. $\sum_{i=1}^m \partial_i \mathcal{H}_1(x) =0$
 and  that  for any $x, v\in\RR^m$,
\be
v^* (\nabla^2\mathcal{H}_1(x))v\geq  \frac{\xi}{m}  \sum_{i, j =1}^m  |v_i-v_j|^2 \; 
\ee
with some positive $\xi>0$.
Then for any  function  of the form $f(\lambda)=F( \sum_{i=1}^mv_i x_i)$, where
$\sum_iv_i=0$  and $F:\RR\to\RR$ is any smooth function, 
 we have
\be\label{lsii}
\int f^2\log f^2\rd\omega-\left(\int f^2\rd\omega\right)\log\left(\int f^2\rd\omega\right)
 \leq \frac{1}{\xi  N}\int|\nabla f|^2\rd\omega.
\ee
\end{lemma}

\begin{proof} In the space $\RR^m$ we 
introduce new  coordinates
 $ (z, w) = M^* (x_1, \ldots, x_m)$ with 
$z= (z_1, \ldots, z_{m-1}) \in \RR^{m-1}$ , 
\be
w:=  m^{-1/2} \sum_i x_i, 
\ee
and $M$ is an orthogonal matrix. 
Since $\mathcal{H}_1(x)$ is independent of $x_1 + \ldots + x_m$, 
we can define  $\wh {\mathcal H}_1 (z) : = {\mathcal H}_1(x)$. 
Similarly, the function $f(\lambda)=F( \sum_{i=1}^mv_i x_i)$ with $\sum_iv_i=0$
depends only on the $z$ coordinates, i.e. it can be written as
$g (z)= f(\lambda)$. 
Hence we can rewrite
$$
\int_{\RR^N} f^2\log f^2\rd\omega = \int_{\RR^{m-1}} g^2\log g^2\rd  \nu,\qquad
\int_{\RR^{N}} f^2\rd\omega= \int_{\RR^{m-1}} g^2\rd\nu,
$$
where $\rd  \nu  = \nu ( z ) \rd z  $ with
$$
 \nu ( z ) := 
\frac{1}{\tilde Z}e^{-N\tilde { \mathcal{H}}(z )}=\frac{ 1 } {Z} \int_{\RR\times \RR^{N-m}}
 e^{-N \mathcal{H}(x,y)}\rd w \rd y.
$$
Introduce the variable  $q = (w,y)\in \RR\times \RR^{N-m}$ and denote by 
 $ \mathcal{H}_{qq}, \mathcal{H}_{zq}, \mathcal{H}_{zz}$ the matrices of second partial derivatives. 
As $\mathcal{H}_2$ is convex, the Brascamp-Lieb inequality yields 
$$
\tilde{ \mathcal{H}}_{zz}    \geq   \frac{ 1 } {Z} \int_{\RR\times \RR^{N-m}}
 e^{-N \mathcal{H}(x,y)}  \Big [  \mathcal{H}_{zz}  -   \mathcal{H}_{zq}    [ \mathcal{H}_{qq}]^{-1} \mathcal{H}_{zq}  \Big ]   \rd w \rd y  .
$$
Since $\mathcal{H}_1$ is independent of $q$, we have, by assumption of the positivity
 of  the Hessian of ${\mathcal H}_2$, that for any $q$ fixed,   
\be
 ( \mathcal{H}_2)_{zz}  -   \mathcal{H}_{zq}    [ \mathcal{H}_{qq}]^{-1} \mathcal{H}_{zq}   =
  ( \mathcal{H}_2)_{zz}   -   (\mathcal{H}_2)_{zq}    [ (\mathcal{H}_2)_{qq}]^{-1} (\mathcal{H}_2)_{zq}    \ge 0.
\ee
Thus we have,  for any $u \in \RR^{m-1}$, that 
\be
u^* \tilde{ \mathcal{H}}_{zz}  u    \geq  u^*  (\hat {\mathcal {H}}_1)_{zz} u 
=   u^* \wt M^* (\mathcal{H}_1)_{xx} \wt M u \ge   \frac  \xi m  \sum_{i, j} [(\wt Mu)_i - (\wt M u )_j]^2 ,
\ee
where $\wt M$ denotes the first $m-1$ columns of $M$. Since the last column of $M$ is parallel with
$(1,1,\ldots, 1)\in \RR^m$ and $M$ is an orthogonal matrix, we have
 $ \sum_i (\wt Mu)_i = 0$ and 
\be
 \frac  \xi m  \sum_{i, j=1}^m [(\wt Mu)_i - (\wt M u )_j]^2   =  2\xi    \sum_{i=1}^m [(\wt Mu)_i ]^2  
= 2 \xi \sum_{i=1}^{m-1} u_i ^2.
\ee
Hence the measure $\nu\sim \exp(-N {\mathcal {\wt H}})$ 
is log-concave with a lower bound $2N\xi$ on the Hessian of $N {\mathcal {\wt H}}$,
 and we can apply the 
Bakry-Emery argument to prove the logarithmic Sobolev inequality for $\nu$.
 Without loss of generality we can assume
that $\int f^2\rd\omega=\int g^2\rd\nu= 1$.
 Therefore,  we have 
\be
\int_{\RR^N} f^2\log f^2\rd\omega =   \int_{\RR^{m-1}} g^2\log g^2\rd  \nu 
 \le \frac{1}{N \xi} \int_{\RR^{m-1}}
 | \nabla_z g |^2 \rd  \nu  =   \frac{1}{N \xi} \int_{\RR^N} | \nabla_x f |^2 \rd  \om,
\ee
where we have used the orthogonality of $M$ to show that  $| \nabla_z g |^2=| \nabla_x f |^2$.
This proves  the  estimate \eqref{lsii}.
\end{proof}

\bigskip 

It is now immediate, from Lemma \ref{lem:localConvexity},  Lemma \ref{lem:localLSI}
and Herbst's lemma, that the following concentration holds.

\begin{corollary}\label{cor:ConcentrationDifferences}
For any function $f(\{\lambda_i,i\in I^{(k,M)}\})=\sum_{I^{(k,M)}}v_i\lambda_i$ with $\sum_i v_i=0$ we have
$$
\P_{\omega^{(k,M)}}(|f-\E_{\omega^{(k,M)}}(f)|>x)
\leq 2\exp\left(-\frac{\beta}{4}\frac{ N^{2-2\epsilon}}{i_M |v|^2}x^2\right).
$$
\end{corollary}

Choosing $v_j = -v_{j+1} = 1$ and all other $v_i$'s being zero, this corollary shows that
the particle differences $\lambda_j-\lambda_{j+1}$ concentrate around their mean with respect to the
$\omega^{(k,M)}$ measure. By choosing $\epsilon$ small and $M$ almost
order one, we obtain concentration
almost up to the optimal scale $1/N$.
If we can justify that the measures $\omega^{(k,M)}$ and $\mu$ are very
close (in a sense to be defined), we will have concentration of differences at the optimal scale for
$\mu$. We will then separately show, by using the loop equation, that accuracy will hold
at the same scale as well. This is the purpose of the next subsection, through an inductive
argument.

\subsection{The induction}\label{subsec:induction}

The purpose of this paragraph is to prove the following proposition: if accuracy holds at scale $N^{-1+a}$, it holds also at scale $N^{-1+\frac{3}{4}a}$.

\begin{proposition}\label{prop:induction}
 Assume that for some $a\in(0,1)$ the following property holds:
for any $\alpha,\epsilon>0$, there are constants
$\delta,c_1,c_2>0$ such that for any $N\geq 1$ and $k\in\llbracket \alpha N,(1-\alpha) N\rrbracket$,
\begin{equation}\label{eqn:IndHyp}
\P_\mu\left(|\lambda_k-\gamma_k|> N^{-1+a+\epsilon}\right)\leq c_1e^{-c_2N^\delta}.
\end{equation}
Then the same property holds also replacing $a$ by $3a/4$:
for any $\alpha,\epsilon>0$, there are constants
$\delta,c_1,c_2>0$ such that for any $N\geq 1$ and $k\in\llbracket \alpha N,(1-\alpha) N\rrbracket$,
we have
$$
\P_\mu\left(|\lambda_k-\gamma_k|> N^{-1+\frac{3}{4}a+\epsilon}\right)\leq c_1e^{-c_2N^\delta}.
$$
\end{proposition}

\noindent{\bf Proof of Theorem \ref{thm:accuracy}.}
This is an immediate consequence of the initial estimate, Proposition \ref{prop:InitialEstimate}, and iterations
of Proposition \ref{prop:induction}.\hfill\qed

Two steps are required in the proof of the above Proposition \ref{prop:induction}.
First we  will prove that   concentration holds  at the smaller scale $N^{-1+\frac{a}{2}}$.

\begin{proposition}\label{prop:concentration}
Assume that (\ref{eqn:IndHyp}) holds. Then for any $\alpha>0$ and $\epsilon>0$,
there are constants $c_1,c_2,\delta>0$ such that for any $N\geq 1$ and $k\in\llbracket \alpha N,(1-\alpha)N\rrbracket$,
$$
\P_\mu\left(|\lambda_k-\E_\mu(\lambda_k)|>\frac{N^{\frac{a}{2}+\epsilon}}{N}\right)\leq
c_1e^{-c_2N^\delta}.
$$
\end{proposition}

The above step builds on the locally constrained measures of the previous subsection.
Then, knowing this better concentration, the accuracy can be improved to the scale $N^{-1+\frac{3a}{4}}$.

\begin{proposition}\label{prop:accuracy}
Assume that (\ref{eqn:IndHyp}) holds.
Then for any $\alpha>0$ and $\epsilon>0$,
there is a constant $c>0$ such that for any $N\geq 1$ and
$k\in\llbracket \alpha N,(1-\alpha)N\rrbracket$,
$$
\left|\gamma_k^{(N)}-\gamma_k\right|\leq c\frac{N^{\frac{3a}{4}+\epsilon}}{N}.
$$
\end{proposition}

Proposition \ref{prop:induction} is an immediate consequence of the last  two  propositions. The proofs of these  two propositions are postponed to the end of this section, after
the following necessary series of lemmas.

\begin{lemma}\label{lem:concGapsOmega} Take any $\epsilon>0$ and $\alpha>0$. There are constants $c_1,c_2>0$ such that for any
$N\geq 1$, any integers $1\leq M_1\leq M\leq \alpha N$, any $k\in\llbracket\alpha N,(1-\alpha)N\rrbracket$,
and $\omega^{(k,M)}$ from Definition \ref{def:locallyConstrained} associated with $k,M,\epsilon$,
$$
\P_{\omega^{(k,M)}}\left(\left|\lambda_k^{[M_1]}-\lambda_k^{[M]}
-\E_{\omega^{(k,M)}}\left(\lambda_k^{[M_1]}- \lambda_k^{[M]}\right)\right|>\frac{x}{N^{1-\epsilon}}\sqrt{\frac{M}{M_1}}\right)\leq c_1e^{-c_2 x^2}.
$$
\end{lemma}
\begin{proof}
Note that
$\lambda_k^{[M_1]}-\lambda_k^{[M]}$ is of type $\sum_{I^{(k,M)}}v_i\lambda_i$ with $\sum v_i=0$ and
$$|v|^2=\sum_{1}^{M_1}\left(\frac{1}{M_1}-\frac{1}{M}\right)^2+\sum_{M_1+1}^M\frac{1}{M^2}
\leq \sum_{1}^{M_1}\left(\frac{2}{M_1^2}+\frac{2}{M^2}\right)+\sum_{M_1+1}^M\frac {2}{M^2}
\leq \frac{4}{M_1}.$$
Together with  Corollary \ref{cor:ConcentrationDifferences}, this concludes the proof.
\end{proof}

\begin{lemma}\label{lem:diffExpectations}
Assume that for $\mu$ accuracy  holds at scale $N^{-1+a}$, i.e. (\ref{eqn:IndHyp}). Take arbitrary $\alpha,\epsilon>0$.
There exist constants $c_1,c_2,\delta>0$ such that for any
$N\geq 1$, for any integer $M$ satisfying $N^a\leq M\leq \alpha N/2$, for any $k\in\llbracket\alpha N,(1-\alpha)N\rrbracket$
and for any  $j\in\llbracket 1,N\rrbracket$, we have
$$
|\E_\mu(\lambda_j)-\E_{\omega^{(k,M)}}(\lambda_j)|\leq c_1 e^{-c_2 N^{\delta}},
$$
where the measure $\omega^{(k,M)}$ is defined in (\ref{eqn:omega}) from Definition \ref{def:locallyConstrained}  with parameters $k,M,\epsilon$.
\end{lemma}
\begin{proof}
First, the total variation norm is bounded by the square root of the entropy, and by (\ref{eqn:BPS2}) the particles are bounded with very high probability, so we have
$$
|\E_\mu(\lambda_j)-\E_{\omega^{(k,M)}}(\lambda_j)|\leq \tilde C  \sqrt { S(\mu\mid \omega^{(k,M)})} +\OO(e^{-\tilde c N})
$$
for some  $\tilde c,\tilde C>0$ independent of $k, j$.
For the measures we are interested in, using the logarithmic Sobolev inequality for $\mu$, we have for some $c,C>0$
$$
S(\mu\mid \omega^{(k,M)})
\leq
C N^c\,\E_\mu\left(\theta'\left(\frac{(\lambda_{k+M}-\lambda_{k-M})N^{1-\epsilon}}{i_M}\right)^2\right).
$$
Now, as $\theta''(x)=0$ if $|x|<1$ and $\theta'(x)^2<4 x^2$, for some new and universal constants $c,C>0$
\begin{align}\label{eqn:entropyBound}
S(\mu\mid \omega^{(k,M)})
&\leq
C N^c\, \E_\mu\left((\lambda_{k+M}-\lambda_{k-M})^2
{\bf 1}_{|\lambda_{k+M}-\lambda_{k-M}|>\frac{i_M}{N^{1-\epsilon}}}\right)\nonumber\\
&\leq C N^c\,
\Big[ \E_\mu\left((\lambda_{k+M}-\lambda_{k-M})^4\right)\Big]^{1/2}
\Big[\P_\mu\left(
|\lambda_{k+M}-\lambda_{k-M}|>\frac{i_M}{N^{1-\epsilon}}\right)\Big]^{1/2}.
\end{align}
This moment of order 4 is polynomially bounded, for example just by concentration of order $N^{-1/2}$ for all $\lambda_j$'s under $\mu$.
Concerning the above probability, as $|\gamma_{k+M}-\gamma_{k-M}|=\OO(M/N)$, for sufficiently large $N$
if
$|\lambda_{k+M}-\lambda_{k-M}|>\frac{i_M}{N^{1-\epsilon}}$ then either
$|\lambda_{k+M}-\gamma_{k+M}|>M/N^{1-\epsilon}$ or $|\lambda_{k-M}-\gamma_{k-M}|>M/N^{1-\epsilon}$.
But accuracy  holds at scale $N^{-1+a}<M/N$, so both previous events have exponentially small probabilities, uniformly in $k$.
Indeed, one has $k-M,k+M\in\llbracket \alpha N/2,(1-\alpha/2)N\rrbracket$ and by (\ref{eqn:IndHyp})there are constants
$\delta,c_1,c_2>0$ such that for any $N\geq 1$ and $k\in\llbracket \alpha N/2,(1-\alpha/2) N\rrbracket$,
$$
\P_\mu\left(|\lambda_k-\gamma_k|> M/N^{1-\epsilon}\right)\leq c_1e^{-c_2N^\delta}.
$$
This concludes the proof.
\end{proof}

\begin{lemma}\label{lem:toMu}
Assume that for $\mu$ accuracy and concentration hold at scale $N^{-1+a}$. Take arbitrary $\alpha,\tilde\epsilon>0$.
There are constants $c_1,c_2,\delta>0$ such that for any
$N\geq 1$, any integers $N^a\leq M\leq \alpha N$, $1\leq M_1\leq M$, and $k\in\llbracket 2\alpha N,(1-2\alpha)N\rrbracket$,
$$
\P_{\mu}\left(\left|\lambda_k^{[M_1]}-\lambda_k^{[M]}
-\E_{\mu}\left(\lambda_k^{[M_1]}-
\lambda_k^{[M]}\right)\right|>\frac{N^{\tilde\epsilon}}{N}\sqrt{\frac{M}{M_1}}\right)\leq c_1e^{-c_2 N^\delta}.
$$
\end{lemma}
\begin{proof}
Consider the measure $\omega^{(k,M)}$ associated with the choice $\epsilon=\tilde\epsilon/2$.
First note that, by Lemma \ref{lem:diffExpectations},
$$
\left|\E_{\mu}\left(\lambda_k^{[M_1]}-
\lambda_k^{[M]}\right)-\E_{\omega^{(k,M)}}\left(\lambda_k^{[M_1]}-
\lambda_k^{[M]}\right)\right|<ce^{-CN^{\delta_1}},
$$
for some coefficients $c,C,\delta_1$, uniformly in $N,k,M,M_1$.
As a consequence of this exponentially small difference of expectations, the probability bound to prove is equivalent to
the existence of $c_1,c_2,\delta>0$ such that
$$
\P_{\mu}\left(A\right)\leq c_1e^{-c_2 N^\delta},\ A=\left\{
\left|\lambda_k^{[M_1]}-\lambda_k^{[M]}-\E_{\omega^{(k,M)}}\left(\lambda_k^{[M_1]}-
\lambda_k^{[M]}\right)\right|>\frac{N^{\tilde\epsilon}}{N}\sqrt{\frac{M}{M_1}}
\right\},
$$
with the same uniformity requirements. By Lemma \ref{lem:concGapsOmega}, we know that there are such constants with
$$
\P_{\omega^{(k,M)}}\left(A\right)\leq c_1e^{-c_2 N^\delta},
$$
so the proof will be complete if we can prove that $|\P_{\omega^{(k,M)}}\left(A\right)-\P_{\mu}\left(A\right)|$
is uniformly exponentially small.
By the total variation\slash entropy inequality we have:
\be
|\P_{\omega^{(k,M)}}\left(A\right)-\P_{\mu}\left(A\right)|
\leq
\int|\rd\omega^{(k,M)}-\rd\mu|\leq
\sqrt{2S (\mu\mid \omega^{(k,M)})}.
\label{entropyineq}
\ee
This entropy was shown to be exponentially small in the proof of Lemma \ref{lem:diffExpectations}, see equation (\ref{eqn:entropyBound}).
\end{proof}

\begin{lemma}\label{lem:concA2}
Assume that for $\mu$ accuracy and concentration hold at scale $N^{-1+a}$.
For any $\tilde\epsilon>0$ and $\alpha>0$, there are constants $c_1,c_2,\delta>0$ such that for any $N\geq 1$
and $k\in\llbracket 2\alpha N,(1-2\alpha)N\rrbracket$,
$$
\P_\mu\left(\left|\lambda_k-\lambda_k^{[\alpha N]}
-\E_\mu(\lambda_k-\lambda_k^{[\alpha N]}) \right|
> \frac{N^{\frac{a}{2}+\tilde \epsilon}}{N} \right)
\leq
c_1e^{-c_2N^\delta}.
$$
\end{lemma}
Note that in this lemma and its proof, for non-integer $M$ we still
 write $\lambda_k^{[M]}$ for
$\lambda_k^{[\lfloor M\rfloor]}$, where $\lfloor M \rfloor$ means
the lower integer part of $M$.

\begin{proof}
Note first that
\begin{multline*}
\left|\lambda_k-\lambda_k^{[\alpha N]}-\E_\mu(\lambda_k-\lambda_k^{[\alpha N]})
\right|\leq
\left|\lambda_k-\lambda_k^{[N^a]}-\E_\mu(\lambda_k-\lambda_k^{[N^a]})
\right|\\+
\left|\lambda_k^{[N^a]}-\lambda_k^{[\alpha N]}-\E_\mu(\lambda_k^{[N^a]}
-\lambda_k^{[\alpha N]})
\right|.
\end{multline*}
By the choice $M_1=1$, $M=N^a$ in Lemma \ref{lem:toMu}, the probability that the
first term is greater than $\frac{N^{\frac{a}{2}+\tilde \epsilon}}{N}$ is exponentially
small, uniformly in $k$, as desired.
Concerning the second term, given some $r>0$ and $q\in\NN$ defined by $1-r\leq a+q r<1$, it is bounded by
\begin{multline*}
\sum_{\ell=0}^{q-1}
\left|\lambda_k^{[N^{a+(\ell+1)r}]}-\lambda_k^{[N^{a+\ell r}]}
-\E_\mu\left(\lambda_k^{[N^{a+(\ell+1)r}]}-\lambda_k^{[N^{a+\ell r}]}\right)
\right|\\
+
\left|\lambda_k^{[a+qr]}-\lambda_k^{[\alpha N]}
-\E_\mu\left(\lambda_k^{[N^{a+qr}]}-\lambda_k^{[\alpha N]}\right)
\right|.
\end{multline*}
By Lemma \ref{lem:toMu}, for any $\epsilon>0$, each one of these $q+1$ terms has an exponentially small probability of being greater than
$
\frac{N^{\epsilon+\frac{r}{2}}}{N}
$. Consequently, choosing any $\epsilon$ and $r$ (and therefore $q$) such that $\epsilon+\frac{r}{2}<a/2$ concludes the proof.
\end{proof}

\bigskip
\noindent{\bf Proof of Proposition \ref{prop:concentration}.}
We just need to write
$$
|\lambda_k-\E_\mu(\lambda_k)|\leq
\left|\lambda_k-\lambda_k^{[\alpha N]}-
\E_\mu(\lambda_k-\lambda_k^{[\alpha N]})\right|
+
\left|\lambda_k^{[\alpha N]}-\E_\mu(\lambda_k^{[\alpha N]})\right|.
$$
By Lemma \ref{lem:concA2}, the first term has exponentially small probability to be greater than $\frac{N^{\frac{a}{2}+\epsilon}}{N}$. By the logarithmic Sobolev inequality for $\mu$ with constant $\OO(1)$ (see (\ref{eqn:LSImu})), the second term has an even better concentration, at scale $1/N$:
$$
\P\left(|\lambda_k^{[\alpha N]}-\E_\mu(\lambda_k^{[\alpha N]})|>\frac{x}{N}\right)\leq C\, e^{-c x^2}.
$$
This concludes the proof of the proposition.
\hfill\qed
\vspace{0.3cm}

\noindent{\bf Proof of Proposition \ref{prop:accuracy}.}
Thanks to Lemmas \ref{lem:Johansson}, \ref{lem:HS}, and reproducing the proof of
Corollary \ref{cor:InitialAccuracy}, we know it is sufficient to prove that for any $\delta>0$
$$
\frac{1}{N^2}\left|\var\left(\sum_k\frac{1}{z-\lambda_k}\right)\right|
$$
goes uniformly to 0 where $z=E+\ii\eta$, $E\in[A+\delta,B-\delta]$ and
$\eta\geq N^{-1+\frac{3a}{4}+\epsilon}$.

Let $i_0$ be the index in $[0,N]$ such that the typical position $\gamma^N_{i_0}$ is the closest to $E$.
Define the indexes of particles close to $E$, far from $E$ and in the edge as
\begin{align*}
\Int&=\{i:|i-i_0|<N^{a+\epsilon}\},\\
\Ext&=\{i:|i-i_0|\geq N^{a+\epsilon}, i\in\llbracket \alpha N,(1-\alpha)N\rrbracket\},\\
\Edg&=\{i: i\not\in\llbracket \alpha N,(1-\alpha)N\rrbracket\},
\end{align*}
where $\alpha$ is small enough such that
\begin{equation}\label{eqn:alphaGamma}
\gamma_{\alpha N}<A+\frac{\delta}{2}<A+\delta<E<B-\delta<B-\frac{\delta}{2}<\gamma_{(1-\alpha)N}.
\end{equation}
We choose $\alpha_k=\E_\mu(\lambda_k)$ in the following equations.
Then
\begin{multline}\label{eqn:EdgExtInt}
\frac{1}{N^2}\left|\var_\mu\left(\sum_{k}\frac{1}{z-\lambda_k}\right)\right|
\leq
\frac{C}{N^2}\E_\mu\left(\left|\sum_{k\in \Edg}\frac{1}{z-\lambda_k}-\frac{1}{z-\alpha_k}\right|^2\right)\\
+
\frac{C}{N^2}\E_\mu\left(\left|\sum_{k\in \Ext}\frac{1}{z-\lambda_k}-\frac{1}{z-\alpha_k}\right|^2\right)
+
\frac{C}{N^2}\E_\mu\left(\left|\sum_{k\in \Int}\frac{1}{z-\lambda_k}-\frac{1}{z-\alpha_k}\right|^2\right).
\end{multline}
The edge term is bounded by
\begin{equation}\label{e:edge}
\frac{C}{N}\sum_{\Edg}\E_\mu\left(\left|\frac{1}{z-\lambda_k}-\frac{1}{z-\alpha_k}\right|^2\right)
\leq
\frac{C'}{N\eta^2}\sum_{\Edg}\P\left(|E-\lambda_k|<\frac{\delta}{3}\right)
+
\frac{C'}{N\delta^2}\sum_{\Edg}\E_\mu\left(|\lambda_k-\alpha_k|^2\right).
\end{equation}
{F}rom the condition (\ref{eqn:alphaGamma}) and the large deviation estimate
 (\ref{eqn:largDev1}), the above probability is exponentially small. Moreover,
 the above $L^2$ moments are $\OO(1/N)$
by the spectral gap inequality for $\mu$, see e.g. equation (\ref{eqn:initialconcentration}).
 Hence the edge term goes to 0 uniformly.

 Using the accuracy  at scale $N^{-1+a}$ and
the concentration at scale $N^{-1+a/2}$ (Proposition \ref{prop:concentration}),
the second term in (\ref{eqn:EdgExtInt})
is bounded,
up to constants, for some $c_1,c_2,\delta>0$ by
\begin{multline*}
\frac{1}{N^2}\E\left(\left|\sum_{k\geq N^{a+\epsilon}}\frac{\lambda_{i_0+k}-\alpha_{i_0+k}}
{(\frac{k}{N})^2}\right|^2\right)+c_1 e^{-c_2 N^\delta}\\
\leq
N^2\E\left(\sum_{k\geq N^{a+\epsilon}}
\frac{|\lambda_{i_0+k}-\alpha_{i_0+k}|^2}{k^2}\right)\sum_{k\geq N^{a+\epsilon}}
\frac{1}{k^2}+c_1 e^{-c_2 N^\delta}
\leq
N^2N^{-2+a}(N^{-a})^2=N^{-a}.
\end{multline*}
In particular, it converges uniformly to 0.
For the third term, for some $c>0$ it is less than
\begin{multline*}
\frac{1}{N^2\eta^4}\E\left(\left(\sum_{Int}|\lambda_k-\alpha_k|\right)^2\right)
\leq
\frac{c}{N^2\eta^4}N^{a+\epsilon}\E\left(\sum_{Int}|\lambda_k-\alpha_k|^2\right)\\
\leq c\frac{(N^{a+\epsilon})^2}{N^2\eta^4}N^{-2+a}=c\frac{N^{3a+2\epsilon}}{N^4\eta^4}.
\end{multline*}
This goes to 0 if $\eta\gg N^{-1+\frac{3a}{4}+\frac{\epsilon}{2}}$, concluding the proof.
\hfill\qed

\section{Local equilibrium measure}\label{sec:loceq}

\subsection{Construction of the local measure}

Let $0<\kappa<1/2$. Choose $q\in [\kappa, 1-\kappa]$ and set $L=[Nq]$ (integer part).
Fix an integer $K$ with $K\le (N-L)/2$, in fact we will always assume
that $K$ depends on $N$ as $K=N^k$ with $k<1$.
 We will study the local spacing statistics
of $K$ consecutive particles
$$
  \{ \lambda_j\; : \; j\in I\}, \qquad I=I_L:=
 \llbracket L+1, L+K \rrbracket.
$$
These particles are typically located near
$E_q$ determined by the relation
$$
  \int_{-\infty}^{E_q} \rho(t) \rd t = q.
$$
Note that $|\gamma_L- E_q|\le C/N$.

We will distinguish the inside and outside particles
by renaming them as
\be\label{35}
(\lambda_1, \lambda_2, \ldots,
\lambda_N):=(y_{1}, \ldots y_{L}, x_{L+1},  \ldots x_{L+K}, y_{L+K+1},
  \ldots y_{N}) \in \Xi^{(N)},
\ee
but note that they keep their original indices.
The notation $\Xi^{(N)}$ refers to the simplex
$\{\bz \; :\; z_1<z_2< \ldots < z_N\}$ in $\RR^N$.
In short we will write
$$
\bx=( x_{L+1},  \ldots x_{L+K} ), \qquad \mbox{and}\qquad
 \by=
 (y_{1}, \ldots y_{L}, y_{L+K+1},
  \ldots y_{N}),
$$
all  in increasing order, i.e. $\bx\in \Xi^{(K)}$ and
$\by \in \Xi^{(N-K)}$.
We will refer to the $y$'s as the  external
points and to the $x$'s as  internal points.

We will fix the external points (often called
as boundary conditions) and study
the conditional measures on the internal points.
We consider the parameters $L$ and $K$ fixed and we
will not indicate them in the notation.
We first define  the
{\it local equilibrium measure} on $\bx$ with boundary condition  $\by$ by
\begin{equation}\label{eq:muyde}
 \quad
\mu_{\by} (\rd\bx)  = u_\by(\bx) \rd \bx, \qquad
u_\by(\bx):=  u (\by, \bx) \left [ \int u (\by, \bx) \rd \bx \right ]^{-1},
\end{equation}
where $u$ is the density of $\mu_V$.
Note that for any fixed $\by\in \Xi^{(N-K)}$,  $x_j$
lies in the interval $[y_{L}, y_{L+K+1}]$.

Given the classical locations, $\gamma=\{\gamma_1,\gamma_2, \ldots , \gamma_N\}$
with respect to the $\mu$-measure, we define the {\it relaxation measure}
 $\mu_N^{\tau, \gamma} =\mu^\tau$  by
\be\label{232}
\rd\mu^\tau :=\frac{Z}{ Z_{\mu^\tau}}e^{-N Q^\tau }\rd\mu,
 \quad Q^\tau (\bx) =    \sum_{j\in I}
Q_j^\tau(x_j)     , \qquad Q_j^\tau (x) = \frac{1}{2 \tau } (   x -   \gamma_j)^2,
\ee
where $Z_\mu$ is chosen such that $\mu$ is a probability measure.
Here $0< \tau < 1$ is a parameter which may even  depend
on $\by$, i.e., $\tau=\tau(\by)$ is allowed. Note that an artificial
quadratic confinement has been added to the equilibrium measure.
We define {\it the local relaxation measure}
$\mu_{\by}^\tau$ to be conditional measure of $\mu^\tau$.

Define the  Dyson Brownian motion
 reversible with respect to $\mu_{\by}^\tau$,
by the Dirichlet form
\be
D_{\mu_{\by}^\tau}(f) =
  \sum_{i\in I} \frac{1}{2N}
\int (\partial_i f)^2 \rd \mu_{\by}^\tau ,
\label{def:dirmu}
\ee
where $\partial_i=\partial_{x_i}$.
The Hamiltonian $\cH_\by^\tau$ of the measure $\mu_\by^\tau (\rd \bx) \sim \exp(-N\cH_\by^\tau)$
 is given by
\be\label{24}
\cH_{\by}^\tau (\bx) =
 \sum_{i\in I}  \frac{\beta}{2}V_\by (x_i)
-  \frac{ \beta }{N} \sum_{i,j\in I\atop i< j}
\log |x_{j} - x_{i}|  +    \sum_{i \in I} Q_i^\tau  (x_i ),
\ee
\be
V_\by (x) = V(x) - \frac{ 1}{ N} \sum_{j \not \in I}
\log |x - y_{j}|.
\ee

We now define the set of {\it good boundary configurations}
with a parameter $\e_0>0$ and a parameter $\delta=\delta(N)>0$ that in the applications
may depend on $N$:
\begin{align}\label{goodset}
  \cG_{\delta,\e_0}=\cG:=
\Big\{  & \by \in \Xi^{(N-K)}\; :\; |y_k-\gamma_k|\le \delta, \; \forall \, k\in
\llbracket N\kappa/2, L\rrbracket \cup \llbracket L+K+1,
N(1-\kappa/2)\rrbracket,  \\
& \mbox{and}\;\;  |y_k-\gamma_k|\le 1, \; \forall \, k\in \llbracket 1, N\rrbracket,  \nonumber \\
& \mbox{and}\;\; \E_{\mu_\by} (x_j-\gamma_j)^2\le \delta^2 \;\;\mbox{for all}
\;\; j\in \llbracket L+1, L+K \rrbracket \nonumber\\
& \mbox{and}\;\; y_L-y_{L-1}\ge \exp(-N^{\e_0}), \;\;\; y_{L+K+2}-y_{L+K+1}\ge \exp(-N^{\e_0})
\Big\} .\nonumber
\end{align}
First we show that the good configurations have overwhelmingly large probability
\begin{lemma}
For any $\e_0>0$ and  for any choice $\delta= N^{-d}$
with $d\in (1-k,1)$, there is an $\e'>0$  depending on $d$  such that
\be
   \P_\mu (\cG^c)\le Ce^{-cN^{\e'}} + Ce^{-cN^{\e_0}} .
\label{goodsetprob}
\ee
\end{lemma}

\begin{proof}
We have proved  in Theorem~\ref{thm:accuracy} that for any choice $\delta= N^{-d}$
with $d\in (0,1)$ the probability that the first condition
in \eqref{goodset} is violated is bounded by $C\exp(-cN^{\e'})$
with some $\e'>0$ depending on $d$.
Similarly, the second condition is violated
with an analogous very small probability by \eqref{eqn:largDev1}.
To check the probability to violate the third requirement in the definition
of $\cG$, we use that
\begin{align}
  \P_\mu \Big\{ \E_{\mu_\by} (x_j-\gamma_j)^2 \ge\delta^2\Big\}
 & \le  \P_\mu \Big\{ \P_{\mu_\by} \{ |x_j-\gamma_j| \ge\delta/2\}\ge 3\delta^2/4
\Big\}+C\exp(-cN^{\e'})\nonumber\\
 & \le  C\delta^{-2} \E_\mu \P_{\mu_\by} \big\{ |x_j-\gamma_j|\ge \delta/2\big\}+C\exp(-cN^{\e'})
 \nonumber\\
& \le   C\delta^{-2} \P_\mu\big\{ |x_j-\gamma_j|\ge \delta/2\big\}
 \le c_1 e^{-c_2 N^{\e'}},
\end{align}
since for $\by$ satisfying the first two conditions of \eqref{goodset} we have
$$
  \E_{\mu_\by} (x_j-\gamma_j)^2 \le \delta^2/4 +
\P_{\mu_\by} \{ |x_j-\gamma_j| \ge\delta/2\}
$$
as $x_j-\gamma_j\le y_{L+K+1} - \gamma_j \le \delta +\gamma_{L+K+1}-\gamma_1\le 1$
and also a similar lower bound holds.

Finally, we show that
$$
  \P_\mu\big( y_L-y_{L-1}\le  \exp(-N^{\e_0})\big)\le  Ce^{-cN^{\e_0}},
$$
and a similar bound holds for the other condition in the fourth
line of \eqref{goodset}. For simplicity of the presentation
and to avoid introducing new notations, we will actually prove
$$
  \P_\mu\big( y_{L+1}-y_L\le  \exp(-N^{\e_0})\big)\le  Ce^{-cN^{\e_0}}
$$
from which the previous inequality follows just by shifting the
indices. With the events
$$
  A:= \big\{ y_{L+1}-y_L\le  \exp(-N^{\e_0})\big\}, \qquad
\Omega:=\big\{ y_{L+K+1}-y_L\le  2a\big\},
$$
we write
\be
  \P_\mu (A) = \E_\mu \big[{\bf 1}(\Omega)\P_{\mu_\by}(A)\big] + \P_\mu(\Omega^c).
\label{decc}
\ee
Choosing $a=C_0 K/N$ with a sufficiently large fixed constant $C_0$
   Theorem~\ref{thm:accuracy} and $\delta \ll K/N$
guarantee that  $\P_\mu(\Omega^c)$ is subexponentially small.

We will prove that
\be\label{tailprob1}
\P_{\mu_\by}  ( x_{L+1} - y_L \le  N^{-2} r ) \le C_Vr
\ee
for any $r\in (0,1)$. The constant depends on $V$, more precisely
\be
  C_V= C + C\sup\big\{ |V'(x)|\; :\; x\in [y_L, y_{L+K+1}]\big\}.
\label{CVdef}
\ee
{F}rom \eqref{tailprob1} the necessary subexponential
estimate on the first term in \eqref{decc} follows by choosing $r=N^{-2}\exp(-N^{\e_0})$.

To prove \eqref{tailprob1},
on the set $\Omega$ we can shift the measure such that $y_L=-y_{L+K+1}$
and denote $a:=-y_L$.
Then we have
\begin{align}
 \int\ldots\int_{-a+ a \varphi }^{a- a \varphi}  & \rd \bx
 \prod_{i,j\in I\atop i < j} (x_i-x_j)^\beta
e^{- N\frac{\beta}{2} \sum_j V_\by (x_j)  } \nonumber \\
 & = (1-\varphi)^{ K+\beta K(K-1)/2} \int\ldots\int_{-a }^{a}   \rd \bw
\prod_{i < j} (w_i-w_j)^\beta e^{- N \frac{\beta}{2}\sum_j V_\by ((1-\varphi) w_j)},
\nonumber
\end{align}
where we set $w_j:=(1-\varphi)^{-1}x_{L+j}$, $\rd \bx =   \rd x_{L+1} \ldots \rd x_{L+K}$
and $\rd\bw =  \rd w_1 \ldots \rd w_K$.
By definition,
\begin{align}
e^{- N \frac{\beta}{2}  V_\by ((1-\varphi) w_j) }
  & = e^{- N \frac{\beta}{2}  V ((1-\varphi) w_j)}
  \prod_{i \le L} ( (1-\varphi) w_j - y_i)^\beta
 \prod_{i \ge L+K+1} ( y_i-(1-\varphi) w_j)^\beta
\nonumber \\
& \ge e^{- N\frac{\beta}{2}  V ( w_j)
- C_V\varphi N  }
(1 - \varphi)^N \prod_{i \le L} ( w_j - y_i)^\beta  \prod_{i \ge L+K+1} ( y_i- w_j)^\beta  .
\nonumber
\end{align}
Note that we only used that $V$ is a $C^1$-function with bounded derivative
in performing a Taylor expansion and using that $w_j\le a$ is finite.
Hence
$$
\frac{1}{Z}\int\ldots \int_{-a+ a \varphi }^{a- a \varphi}
  \rd \bx \prod_{i,j\in I\atop i < j} (x_i-x_j)^\beta e^{- N\frac{\beta}{2}\sum_j V_\by (x_j)}
  \ge (1-\varphi)^{N K+CK^2} e^{- C_V NK \varphi }
$$
with
$$
   Z := \int_{-a }^{a}   \rd \bw \prod_{i,j\in I\atop i < j} (w_i-w_j)^\beta e^{- N
 \frac{\beta}{2}\sum_j V_\by ( w_j)}.
$$
Therefore the $\mu_\by$-probability of $y_{L+1} -y_L = x_{L+1}-y_L\ge a(1-\varphi)$
can be estimated by
$$
\P_{\mu_\by}  ( x_{L+1}  \ge -a+ \varphi a )\ge  (1-\varphi)^{N K+CK^2} e^{- C_V NK \varphi }
\ge 1-  (C_V +C)NK\varphi
$$
by using $K\le N$. Choosing $\varphi=N^{-2}r/a$ and recalling
that $a\sim K/N$, we arrive at \eqref{tailprob1}.
\end{proof}

\begin{proposition}\label{prop:mumu} Let $\varphi>0$ be fixed.
For any smooth, compactly supported
function $G:\RR\to \RR$ we have
\be
\lim_{N\to\infty}\Bigg|\E_\mu  \big [\E_{ \mu_{\by}} -\E_{ \mu_{\by}^\tau}  \big  ]
\frac 1 K \sum_{i \in I} G\Big( N(x_i-x_{i+1}) \Big)\Bigg| =0,
\ee
provided
\be
\frac{1}{2}\hat\tau\le \tau(\by)\le 2\hat \tau
\qquad \mbox{for any $\by\in \cG$}
\label{taubound}
\ee
holds  for the function $\tau=\tau(\by)$
with some constant $\hat\tau=\hat \tau_N$ such that
\be\label{cond1}
\frac{N \delta^2}{ \hat\tau }  \le N^{-\varphi}.
\ee
\end{proposition}

We remark that, with  a slight abuse of notation, the last term, $i=L+K$ in
the sum involving the non-existing $x_{i+1}=x_{L+K+1}$ is defined to be zero.
We also point out that
the notation $\E_\mu \E_{ \mu_{\by}}$ means that
 the law of $\by$ is given by $\mu$ in
 the first expectation and we are using the measure $\mu_\by$ in the second
one. Of course, we have   $\E_\mu=\E_\mu \E_{ \mu_{\by}}$.

\medskip

\begin{proof}
For any configuration $\by$, any $\tau$ (may depend on $\by)$  and
 for any smooth function  $G$ with compact support, we have
\be\label{relax}
   \Bigg| \big [\E_{ \mu_{\by}} -\E_{ \mu_{\by}^\tau}  \big  ]
\frac 1 K \sum_{i \in I} G\Big( N(x_i-x_{i+1}) \Big)\Bigg|
\le C \Big(  \frac {\tau N^{\varphi/2}}{ K}
 D \big  (\mu_{\by} | \mu_{\by}^\tau   \big  ) \,    \Big)^{1/2}
+ Ce^{-cN^{\e/2}} \sqrt{S (\mu_{\by} | \mu_{\by}^\tau   \big  )},
\ee
Here we also  introduced  the notations
\be\label{def:Dirform}
  D(\mu\mid \om) : = \frac{1}{2N}\int
\Big|\nabla \log \Big(\frac{\rd\mu}{\rd\om}\Big)\Big|^2 \rd \mu =
 \frac{1}{2N} \int \Big|\nabla \sqrt{ \frac{\rd\mu}{\rd\om}}\Big|^2\rd \om
\ee
and
$$
  S(\mu\mid \om): = \int \log \Big(\frac{\rd \mu}{\rd\om}\Big) \rd \mu
$$
for any probability measures $\mu, \om$. The estimate \eqref{relax} follows
from our the local relaxation to equilibrium argument
that in this form first appeared in Theorem 4.3 of \cite{ESYY}.
We will neglect the exponentially small entropy term since
it can be estimated by the Dirichlet form, i.e. by the first term
as long as $\tau\ge N^{-C}$.

We thus obtain
\be
\Bigg|\E_\mu  \big [\E_{ \mu_{\by}} -\E_{ \mu_{\by}^\tau}  \big  ]
\frac 1 K \sum_{i \in I} G\Big( N(x_i-x_{i+1}) \Big)\Bigg|
\le
C   \Big(  \frac {N^{\varphi/2}}{ K}  \E_\mu  \big[ {\bf 1}_\cG\;
 \tau(\by) D \big  (\mu_{\by} | \mu_{\by}^\tau   \big  )\big] \,    \Big)^{1/2}
 + Ce^{-cN^{\e'}}.
\label{EE}
\ee
To obtain the estimate \eqref{EE} we separated good and bad configurations;
we used \eqref{relax} for $\by\in \cG$. On the complement $\cG^c$
we just used the trivial estimate on $G$, and this yields
the subexponentially small second term.

Assuming \eqref{taubound},
we have
\be \frac {1}{ K}  \E_\mu  \big[ {\bf 1}_\cG\;
 \tau(\by) D \big  (\mu_{\by} | \mu_{\by}^\tau   \big  )\big] \,
 \le \frac{N}{K}\E^\mu\Big[  {\bf 1}_\cG\;
 \frac{1}{  \tau(\by) } \sum_{j \in I }  (   x_j -   \gamma_j)^2\Big]
\le \frac{N\delta^2}{\hat \tau}
\le N^{-\varphi}
\ee
by \eqref{cond1}.
Inserting this estimate into \eqref{EE}
we completed the proof of the proposition.
\end{proof}

\subsection{Matching the boundary conditions}\label{sec:match}

Suppose we have measures  $\g $ and $ \mu$ with potentials $W$ and $V$
given by \eqref{eqn:measure} with
 densities $\rho=\rho_V$ and $\rho_W$, respectively.
For our purpose $W(x)=x^2$,  i.e, $\g$ is the Gaussian $\beta$-ensemble
and $\rho_W(t) =\frac{1}{2\pi}\sqrt{[4-t^2]_+}$ is the Wigner semicircle law.
Let  the sequence $\gamma_j$ be the classical location for $\mu$ and
the sequence $\theta_j$ be the classical locations for $\g$.

We will match the boundary conditions for the local
measure on $J_\by:=[y_{L}, y_{L+K+1}]$ around  $E_q$
with those of the $\g$ measure.
For definiteness we choose the interval $J'=[\theta_{L'}, \theta_{L'+K+1}]$
with $L' = \frac{1}{2}(N -  K -1)$ as our reference interval. Note
that $J'$ is  symmetric to the origin.
The local density $\rho_V(E_q)$ at the point $E_q$ we
look at may be different from the density $\rho_W(0)$ at
the origin. Thus the typical length of $J_\by$, which is
$\gamma_{L+K+1} - \gamma_L \sim [\rho_V(E_q)]^{-1} N^{-1}$, may
not be close to the length of $J'$ which is very close to
$[\rho_W(0)]^{-1} N^{-1} = \pi N^{-1}$, so we will have
to rescale the $\g$ measure by a factor
$$
  s_q    \approx \frac{\rho_V(E_q)}{\rho_W(0)}.
$$
In fact, we need to match not only the interval of classical locations
$\gamma$ with $J'$, but the exact interval $I_\by$. This requires
a $\by$-dependendent scaling factor $s=s(\by)$.

{F}rom now on we assume that $\by$ is a good boundary condition
with a parameter $\delta$ that satisfies
\be\label{dNK}
\frac {\delta N } K \to 0.
\ee
We can shift the coordinates
so that
\be\label{t}
- y_L = y_{L+K+1}.
\ee
 Since our observable is translationally invariant, we will not track
the translation and we assume that \eqref{t} holds.
We define
\be
  s(\by):= \frac{\th_{L'}}{y_L}= \frac{\th_{L'+K+1}}{y_{L+K+1}}, \qquad
  s_q: = \frac{\th_{L'}}{\gamma_L}.
\label{sdef}
\ee
We have
\be\label{s-1}
|s(\by)-  s_q |  =\Big| \frac { \th_{L'}} { y_L}-\frac { \th_{L'}} { \gamma_L}\Big|
 \le C \frac { \delta N } K \to 0
\ee
since
\be
  \theta_{L'}\approx -[\varrho_W(0)]^{-1} \frac{K}{2N}, \qquad
  \gamma_L \approx -[\varrho_V(E_q)]^{-1} \frac{K}{2N}, \qquad
  y_L \approx -[\varrho_V(E_q)]^{-1} \frac{K}{2N}, \qquad
\label{appr}
\ee
by using $\by\in \cG$ and \eqref{dNK}. Similar formulas hold
for $\th_{L'+K+1}$, $\gamma_{L+K+1}$ and $y_{L+K+1}$ at
the upper edge of the interval. Here the $A\approx B$
is understood in the sense that the approximation error
at most of order $(K/N)^2$, recalling that $K=o(N)$.

For simplicity of the presentation, we can first shift the original $\mu$-ensemble
such that $E_q=0$. Second, we can perform
 an initial rescaling of the Gaussian $\beta$-ensemble so that $s_q=1$.

\begin{lemma} Assuming $E_q=0$, $s_q=1$, we have
\be\label{38}
 |\gamma_{L+j}-\th_{L'+j}|\le C\frac{j^2}{N^2} + C\delta, \qquad |j|\le \frac{1}{100} N\kappa.
\ee
\end{lemma}

\begin{proof}
The classical locations  $\gamma_{L+j}$ and $\th_{L'+j}$ are given by the equation
\be\label{jN}
\int_{\gamma_L}^{\gamma_{L+j}} \rho_V(x)  = \frac {  j} N,
\qquad \int_{\th_{L'}}^{\th_{L'+j}} \rho_W(x)  = \frac {  j} N.
\ee
We will use the approximations
\be\label{ap}
\rho_V(x) = \rho_V(0)  +   O( x), \qquad \rho_W(x) = \rho_W(0)  +   O( x)
\ee
for small $x$ (to stay away from the spectral edge).
Since $ |y_j - \gamma_j |\le \delta$, we have
\be
\frac {  K+1} N =  \int_{\gamma_L}^{\gamma_{L+K+1}} \rho_V(x)  \rd x
 = \int_{y_L}^{y_{L+K+1}} \rho_V(x)  \rd x + O(\delta)
 =  \rho_V(0)  (y_{L+K+1}-y_L) +  O \left ( \frac {K^2} { N^2} \right ) + O(\delta).
\ee
Similarly,
\be
\frac {  K+1} N =  \int_{\th_{     L'}}^{\th_{     L'+K+1}} \rho_W (x)  \rd x
  =  \rho_W(0)  (\th_{     L'+K+1} - \th_{L'} ) +  O \left ( \frac {K^2} { N^2} \right ).
\ee
Since $- y_L = y_{L+K+1}=- \th_{     L'}/s  = \th_{     L'+ K + 1}/s$
which is comparable with $K/N$ by \eqref{appr},
and since $|s-  1|    \le C \frac { \delta N } K$  from \eqref{s-1}, we have
\be\label{VW0}
| \rho_W(0) - \rho_V(0)  | \le \frac { C K } N + \frac {C\delta N } K.
\ee
{F}rom \eqref{jN}, \eqref{ap} and \eqref{appr} we get
$$
  \varrho_V(0)(\gamma_{L+j}-\gamma_L) + O\Big(\frac{j^2}{N^2}\Big)
  = \frac{j}{N} =  \varrho_W(0)(\th_{L'+j}-\th_{L'}) + O\Big(\frac{j^2}{N^2}\Big),
$$
which combining with \eqref{VW0} and $\rho_W(0)\ge c$ gives
$$
    \gamma_{L+j}-\gamma_L
  = \th_{L'+j}-\th_{L'} + O\Big(\frac{j^2}{N^2}\Big) + O(\delta).
$$
Since $\gamma_L=\th_{L'}$, this completes the proof of the lemma.
\end{proof}

\newcommand{\bt}{\mbox{\boldmath $\theta$}}
\newcommand{\htau}{{\hat\tau}}

\subsection{Rescaling of the reference problem}

Throughout this section we fix  a good boundary configuration.
 $\by\in \cG$ and a number $\tau(\by)$ depending on
this configuration and satisfying \eqref{taubound}.
 We will approximate the local relaxation measure $\mu_\by^\tau$
on $[y_L, y_{L+K+1}]$ by a fixed reference measure.

Given the collection of classical locations $\th_j$ corresponding
to the Gaussian potential $W(x)=x^2$  we define a {\it reference
local relaxation measure} $\g_\th^\htau$ via the  Hamiltonian
\be\label{241}
\cH_{\th}^\htau (\bx) =
 \sum_{i \in I'}  \Big [  \frac{\beta}{2}   x_i^2   - \frac{ \beta }{N} \sum_{j \not \in I'}
\log |x_i - \th_{j} | \Big ]
-  \frac{ \beta }{N} \sum_{ i,j \in I'\atop i<j }
\log |x_{j} - x_{i}|  +   \frac{1}{2 \htau } \sum_{ i \in I'} (   x_i -   \theta_i )^2,
\ee
on the set $[\th_{L'}, \th_{L'+K+1}]$
where  $I':= \llbracket L'+1, L'+K\rrbracket$.
Note that if  $\g$ is the equilibrium measure given by \eqref{eqn:measure} corresponding to $W$
and $\g^\htau$ denotes the corresponding relaxation measure
\be\label{232sigma}
\rd\g^\htau :=\frac{Z}{ Z_{\g^\htau}}e^{-N Q^\htau }\rd\g,
 \quad Q^\htau (\bx) =    \sum_{j\in I'}
Q_j^\htau(x_j)     , \qquad Q_j^\htau (x) = \frac{1}{2 \htau } (   x -   \th_j)^2,
\ee
defined analogously to \eqref{232}, then $\g_\th^\htau$ is the conditional
measure of $\g^\htau$ under the condition that the outside points
are exactly at their classical locations, i.e. $\lambda_j=\th_j$, $j\not\in I'$.

We make three simplifications in the presentation. First, as already in
Section~\ref{sec:match}, we assume that both the  configuration
space $[y_L, y_{L+K+1}]$ for the original measure $\mu_\by^\htau$
and the configuration space $[\theta_{L'}, \theta_{L'+K+1}]$ of the reference
measure $\g_\th^\htau$ are symmetric around the origin; this can be achieved
by an irrelevant shift. Second, we assumed $s_q=1$, which can be achieved
by an irrelevant rescaling of $W$.
Finally, we will set $L'=L$.
This last assumption expresses an irrelevant shift in the labelling
of one of the ensembles. Strictly speaking, shifting would mean that
the original set of particles indices $\llbracket 1, N\rrbracket$ gets shifted as well.
However, in our argument this shift does not play any active role;
the only information we use about the set of indices is that $L$ is
macroscopically separated from its boundary and that its cardinality is $N$.

We now rescale the measure $\g_\theta^\htau$ from $[\theta_{L}, \theta_{L+K+1}]
 =[\theta_{L}, -\theta_{L}] $ to $[y_L, y_{L+K+1}]= [y_{L}, -y_{L}]$
by the factor $s=s(\by)$ defined in \eqref{sdef} (note that $y_L, \th_L<0$).
With the rescaled  boundary conditions
$\theta_j\to\th_j':= \theta_j/s$,
we define the {\it reference local relaxation measure}, or {\it reference measure}
in short, to be
\be\label{def:ref}
   \g_{\th}^{\htau, s}: =\frac { 1 } { Z^{\htau,  \th, s } }
 e^{-N\cH_{\th}^{\htau, s} (\bx)}\rd \bx,
\ee
 a measure on the set
$[y_L, y_{L+K+1}]$
with Hamiltonian
\be\label{242}
\cH_{\th}^{\htau, s} (\bx) =
 \sum_{i \in I}  \Big [  \frac {   \beta s^2  x_i^2 } 2  - \frac{ \beta}{N} \sum_{j \not \in I}
\log |  x_i - \th_{j}/s | \Big ]
-  \frac{ \beta }{N} \sum_{ i,j \in I\atop i<j }
\log | x_{j} -   x_{i}|  +    \frac{s^2}{2 \htau } \sum_{ i \in I} (    x_i -   \theta_i/s )^2.
\ee
The rescaled potential associated with this Hamiltonian is $ W_s(x)= s^2 x^2$.

\newcommand{\non}{\nonumber}

For any smooth  function $G$ with compact support, we have
\be\label{diff}
\E_{ \g_{\th}^{\htau, s} }    \frac 1 K \sum_{i \in I} G\Big( N(x_i-x_{i+1}) \Big)
=  \frac { 1 } { Z^{\htau,  \th, s } }
\int_{\th_L/s}^{- \th_L/s}  \rd \bx \; e^{- N \cH_{\th}^{\htau,s}
(\bx)} \; \frac 1 K \sum_{i \in I} G\Big( N(x_i-x_{i+1}) \Big),
\ee
where $\int_a^b \rd \bx$ stands for the $K$-dimensional integral
$\int_{[a,b]^K}\rd x_{L+1}\ldots  \rd x_{L+K}$ and $ Z^{\htau,  \th, s }$ is the normalization
factor.
Let $x_j = w_j/s$, then the right side becomes
\begin{align}
 \frac { 1 } { Z^{\htau,  \th} } \int_{\th_L }^{- \th_L}
 \rd\bw \; e^{- N \cH_{\th}^{\htau} (\bw)}
\frac 1 K \sum_{i \in I} G\Big( \frac { N(w_i-w_{i+1})} s  \Big)
& = \E_{ \sigma_{\th}^\htau }  \frac 1 K \sum_{i \in I}  G\Big( \frac{ N  (x_i-x_{i+1})}{s}
 \Big )\label{44} \\
 \non
& = \E_{ \sigma_{\th}^\htau}  \frac 1 K \sum_{i \in I}  G\Big( N (x_i-x_{i+1}) \Big)  + o(1),
\end{align}
where we renamed the $w$-variables to $x$-variables in the first step and
in the second step we have used that
\[
\Big| G\Big( N (x_i-x_{i+1})/s \Big) - G\Big( N  (x_i-x_{i+1}) \Big)  \Big|\le
 C|1-s| \| G'\|_\infty
\]
by Taylor expansion and from the fact that $G$ is compactly supported.
Clearly, the difference vanishes as long as $s \to 1$.
Thus we are free to scale the measure with factor converging to  1.
The condition $s\to 1$ will be guaranteed by \eqref{s-1}.

Our main result is the following theorem.

\begin{theorem}\label{thm:mi} Let $0<\varphi\le \frac{1}{38}$.
 Fix $K=N^k$, $\delta = N^{-d}$, $\hat \tau = N^{-t}$
with $d=1-\varphi$, $t=2d-1-\varphi=1-3\varphi$ and $k=\frac{39}{2}\varphi$,
in particular  such that \eqref{cond1},
\eqref{dNK} are satisfied.
Then
\be\label{381}
\Bigg| \E_\mu \E_{ \mu_{\by}^{\htau} } \frac 1 K \sum_{i \in I} G\Big( N(x_i-x_{i+1}) \Big)
   -\E_{ \g_{\th}^{\htau}}
 \frac 1 K \sum_{i \in I} G\Big( N(x_i-x_{i+1}) \Big)\Bigg|
\to 0
\ee
as $N \to \infty$
for any smooth and compactly supported test function $G$.
 Here the law of $\by$ is given by   $\mu$ in the expectation.

\end{theorem}

\begin{proof}
{F}rom the rescaling estimates, \eqref{diff}-\eqref{44},
it suffices to prove that
\be\label{372}
\E_\mu  \big [\E_{ \mu_{\by}^{\htau} } -\E_{ \g_{\th}^{\htau, s(\by)}}  \big  ]
\frac 1 K \sum_{i \in I} G\Big( N(x_i-x_{i+1}) \Big)
\to 0
\ee
as $N \to \infty$. Notice that after the rescaling
both measures $\mu_{\by}^{\htau}$ and $ \g_{\th}^{\htau, s(\by)}$
live on the same interval $[y_L, y_{L+K+1}]$.
In Proposition~\ref{prop:mumu} we already showed that
\be\label{371}
\E_\mu  \big [\E_{ \mu_{\by}^{\htau} } -\E_{ \mu_{\by}^{\htau/s(\by)^2} }
  \big  ]
 \frac 1 K \sum_{i \in I} G\Big( N(x_i-x_{i+1}) \Big)
\to 0,
\ee
since \eqref{s-1}   with $s_q=1$ and
 $\delta N/K\to 0$
guarantee that $\tau(\by):= \htau/s(\by)^2$ satisfies \eqref{taubound}.
Thus the limit \eqref{372} will follow
from the following Proposition that we will prove
in Sections~\ref{sec:conv}:
\begin{proposition}\label{prop:musig}
Under the assumptions of Theorem~\ref{thm:mi}, we have
\be\label{37}
\E_\mu  \big [\E_{ \mu_{\by}^{\htau/s(\by)^2} }
 -\E_{ \g_{\th}^{\htau, s(\by)}}  \big  ]
 \frac 1 K \sum_{i \in I} G\Big( N(x_i-x_{i+1}) \Big)
\to 0.
\ee
\end{proposition}
This  completes the proof of Theorem~\ref{thm:mi}.
\end{proof}

\bigskip

{\bf Proof of Theorem~\ref{thm:Main}.}
Finally, combining Theorem~\ref{thm:mi} with
Proposition~\ref{prop:mumu} and noticing that \eqref{cond1}
is satisfied since $t=2d-1-\varphi$, we have
\be\label{3811}
\Bigg| \E_\mu  \frac 1 K \sum_{i \in I} G\Big( N(x_i-x_{i+1}) \Big)
   -\E_{ \g_{\th}^{\htau}}
 \frac 1 K \sum_{i \in I} G\Big( N(x_i-x_{i+1}) \Big)\Bigg|
\to 0
\ee
as $N \to \infty$. This holds for $K=N^k$ with any $0<k\le\frac{1}{2}$
by selecting a suitable $\varphi$ in Theorem~\ref{thm:mi}.
However, the measure $\sigma_\theta^{\htau}$ is independent of $V$,
the only information we used was that the local density matches.
So we obtain that any two measures $\mu_{\beta, V}$ and $\mu_{\beta, W}$
have the same local gap statistics assuming that the local
densities of the two ensembles coincide.
\qed

\section{Comparison with the reference problem}\label{sec:conv}

In this section we prove Proposition~\ref{prop:musig}.
On the set $\by\in \cG^c$ with subexponentially small
probabality \eqref{goodsetprob} a trivial estimate on $G$
suffices.  For the sequel we therefore assume that $\by\in \cG$ and
we set $\tau(\by):= \htau/s(\by)^2$ which clearly satisfies \eqref{taubound}.
In the first step we will soften the boundary condition $\by$
for the measure local relaxation measure $\mu_\by^\tau$.

\subsection{Regularizing the boundary conditions}

We know that the boundary condition
 $\by\in \cG$ is regularly spaced on the scale
$\delta= N^{-d}$, but it does not exclude that $N\delta = N^{1-d}\gg 1$
points of the colletion $\by$ pile up near the edges
of the interval
$[y_L, y_{L+K+1}]$. This would substantially influence
the local relaxation measure $\mu_\by^\tau$ near the corresponding edge
inside $[y_L, y_{L+K+1}]$. We therefore first replace the
boundary conditions near the edges by the regularly spaced
ones given by $\theta'=\theta/s(\by)$. This change will be controlled only
in the entropy sense. The local relaxation measure with regularized
boundary conditions will then be compared with
the reference measure in the stronger Dirichlet form sense.

Set a parameter
\be\label{def:B}
  B = N^b \quad \mbox{with} \quad  1+\varphi -d \le b < k,
\ee
in particular $\delta N\ll B\ll K$.
Given a boundary condition $\by\in\cG$,
 we define a new boundary condition $\by^B = \{ y_i^B\; : \; i\not\in I\}$ as
\be
 y^B_i: = \left\{
\begin{array}{lll}
 \max \{ \th_i', y_{L-4B}\} & \mbox{for} &  L - 4B \le i \le L \cr
   y_i  & \mbox{for} &  i < L -4B, \quad \mbox{or}\quad i> L+K+4B\cr
   \min\{ \th_i', y_{L+K+4B}\}  & \mbox{for} &  L +K+1 \le i \le L+K+4B,
\end{array}\right.
\ee
i.e., we replace at most $4B$ boundary
conditions $y_i$ with the rescaled classical ones $\th_i'=\th_i/s(\by)$
near the edges of the interval $[y_L,y_{L+K+1}]=[\theta_L', \theta_{L+K+1}']$.
Note that the  configuration space is unchanged.
We have
$$
   y_{L-4B}\le\gamma_{L-4B}+ \delta \le \theta'_{L-4B} + CB^2N^{-2}+C\delta \le \theta'_{L-2B},
$$
where we used that $\by\in \cG$ in the first step and \eqref{38} in the second.
In the last inequality we used that $\theta'_{L-2B} - \theta'_{L-4B}\ge cBN^{-1}$
(by regular spacing) and the definition of $B$ from \eqref{def:B}.
Thus we obtain
\be\label{ybth}
y^B_i =   \th'_{  i},  \quad   L - 2B    \le i \le L,
\ee
and similarly at the upper edge.
In other words, we do replace at least $2B$ boundary condition points near the edges with
the classical ones. Although it may happen that a few $y_i^B$ pile up,
but this occurs away from the edges.
The key property of the family $y_i^B$ is the following bound
\be
\#\{ i\; : \; y_i^B\in J \} \le CN|J|
\label{ybregular}
\ee
for any interval $J$ such that $|J|\ge cN^{-1}$
and $c|J|\le \mbox{dist} (J, [y_L, y_{L+K+1}])\le |J|/c$
with some small constant $c$.

\medskip

Consider the {\it regularized local relaxation measure}, which is defined as
the probability measure
\be\label{def:regloc}
\mu^{B,\tau}_\by(\rd \bx) = Z^{-1}e^{- N \cH_\by^{B,\tau}}\rd\bx
\ee
of $K$ ordered points $\bx=(x_{L+1}, \ldots , x_{L+K})$
in $[y_L, y_{L+K+1}]$, with Hamiltonian
\be
\cH_{\by}^{B,\tau} (\bx) :=
 \sum_{i\in I}   \frac{\beta}{2}V^i_{\by^B} (x_i)
-  \frac{ \beta }{N} \sum_{i,j\in I\atop i< j}
\log |x_{j} - x_{i}|  +       \sum_{i \in I} Q_i^\tau  (x_i ),
\ee
with a quadratic confinement $Q_i^\tau(x) = (2\tau(\by))^{-1} (x-\theta'_i)^2$
as in \eqref{242} and $\tau(\by)= \hat \tau/s(\by)^2$.
The potential $V^i$ is given
by
$$
V^i_{\by^B} (x) = V(x) - \frac{ 2 }{ N} \sum_{ j \le  L}
\log |x  - y_{j}^B| - \frac{ 2 }{ N} \sum_{ j \ge  L+K+1}
\log |x  - y_{j}|
 \qquad \mbox{for}\quad L +1\le i \le L + 4B
$$
$$
V^i_{\by^B} (x) = V(x) - \frac{ 2 }{ N} \sum_{ j \le  L}
\log |x  - y_{j}| - \frac{ 2 }{ N} \sum_{ j \ge  L+K+1}
\log |x  - y_{j}|
 \qquad \mbox{for}\quad L +4B+1\le i \le L +K -4B
$$
and
$$
V^i_{\by^B} (x) = V(x) - \frac{ 2 }{ N} \sum_{ j \le  L}
\log |x  - y_{j}| - \frac{ 2 }{ N} \sum_{ j \ge  L+K+1}
\log |x  - y_{j}^B|
 \qquad \mbox{for}\quad L+K-4B+1 < i \le L +K .
$$
In other words, we replace the  boundary condition  $\by $ with
 $\by^B$ for the points  $x_i$ with $L+1\le i \le L + 4B$
at the lower edge and similarly for the other edge.
The boundary conditions for
the middle points $x_i$ with $L+4B+1\le i\le L+K-4B$
remain unchanged. Recalling \eqref{24}, we have in particular
\be
\cH_{\by}^{B,\tau} (\bx)-\cH_{\by}^{\tau} (\bx)
= \frac{2}{N} \sum_{L - 4B \le   j < L}
\sum_{L < i  \le L+4B }
\left [ -   \log  |x_i- y^B_{j} |  +  \log     |x_i- y_{j} |\right ]
+\big( \mbox{Upper edge}\big),
\ee
where {\it (Upper edge)} refers to an analogous term
collecting interactions
near the upper edge.

\medskip

\begin{lemma}\label{lm:SS} Let $\by\in \cG$.
 The relative entropies of the measures
 $\mu_\by^\tau$ and $\mu^{ B,\tau}_\by$ satisfy
\be \label{77}
S( \mu_\by^\tau | \mu^{ B,\tau}_\by   )+S( \mu_\by^{B,\tau} | \mu^{\tau}_\by   )
 \le C B^2 \log N.
\ee
\end{lemma}

\begin{proof} We start with the following lemma that estimates the relative entropy of
any two measures:

\begin{lemma}\label{lm:HH}
Suppose $\mu_i(\rd x)=Z_i^{-1}e^{-H_i}\rd x$, $i= 1, 2$ are probability measures
with Hamiltonians $H_i$ on a common measure space.
Then
\be\label{852}
S(\mu_1 | \mu_2)   \le  \E_{\mu_1} [ H_2 - H_1 ]  + \E_{\mu_2} [       H_1 - H_2 ].
\ee
We also have the inequality
\be
   \E_{\mu_2} [       H_2 - H_1 ] \le \log Z_1- \log Z_2
 \le  \E_{\mu_1} [       H_2 - H_1 ].
\ee
\end{lemma}

\begin{proof}
By Jensen inequality, we have
\begin{align*}
0\le S(\mu_1 | \mu_2)   & = \int \rd \mu_1 \log \left ( \frac {\rd \mu_1} {\rd \mu_2} \right )
=  \int \rd \mu_1   [    H_2 - H_1 ] +  \log \left ( \frac {Z_2} {Z_1} \right ) \\
&
\le  \E_{\mu_1} [ H_2 - H_1 ] - \log   \left [ \int e^{-H_1}   \frac { \rd x }
  {\int e^{-H_2} \rd x }  \right ] \\
&  \le  \E_{\mu_1} [       H_2 - H_1 ]  + \E_{\mu_2} [       H_1 - H_2 ].
\end{align*}
This completes the proof of Lemma~\ref{lm:HH}.
\end{proof}

\medskip

Hence we have $S( \mu_\by^\tau | \mu^{B,\tau}_\by   )  \le \beta\Omega_1 $, where
\begin{align}\label{e4}
\Omega_1  & :  =
\Big ( \E_{\mu^\tau_\by} -  \E_{ \mu^{ B,\tau}_\by} \Big )   \sum_{L - 4B \le   j < L}
\sum_{L < i  \le L+4B }
\left [ -   \log  ( x_i- y^B_{j} )  +  \log     ( x_i- y_{j} )\right ]
+\big(\mbox{Upper edge}\big). \nonumber   \\
\end{align}
Using that $x_i-y_j\ll 1$,
we clearly have
\begin{align}
\Omega_1 &   \le    - \sum_{L - 4B \le   j < L}\sum_{L < i  \le L+4B }
\left [ \E_{\mu^\tau_\by}     \log  ( x_i- y^B_{j} )  +   \E^{ \mu^{B,\tau}_\by}
   \log  ( x_i- y_{j} )  \right ]+\big(\mbox{Upper edge}\big)\nonumber  \\
 &   \le   C B^2 \log N  - B^2
   \E_{ \mu^{B,\tau}_\by}
   \log  ( x_{L+1}- y_L )+\big(\mbox{Upper edge}\big) . \label{Om1}
\end{align}
In the first term we used the trivial estimate
$x_i-y_j^B\ge \th_L' - \th_{L-1}' \ge cN^{-1}$
for any $j<L$.
The second term will be estimated by Lemma~\ref{lm:rep} below
and this completes the estimate for
$S( \mu_\by^\tau | \mu^{ B,\tau}_\by   )$.
The other relative entropy, $S( \mu_\by^{B,\tau} | \mu^{\tau}_\by   )$
can be treated similarly and this
proves Lemma~\ref{lm:SS}.
\end{proof}

\begin{lemma}\label{lm:rep}
Suppose $\tau \ge N^{-1}$, then for any $p\ge 1$ we have
\be\label{trivgap}
\E_{\mu_\by^\tau}  |\log( x_{L+1} - y_L)|^p\le C_p\log N
\ee
and the same estimate holds  w.r.t the measure $\mu_\by^{ B,\tau}$.
\end{lemma}

\begin{proof}
We will need that
\be\label{tailprob}
\P_{\mu_\by^\tau}  ( x_{L+1} - y_L \le  N^{-3} r ) \le Cr
\ee
for any $r\in (0,1)$. Then \eqref{trivgap} follows from integrating in $r$ from 0 to $1$
and treating the regime $ x_{L+1}-y_L\ge N^{-3}$ trivially
by using $ x_{L+1}-y_L\le y_{L+K+1}-y_L\le CK/N\le 1$.

The estimate \eqref{tailprob} can be proven  essentially
in the same way as \eqref{tailprob1}, just the potential $\frac{\beta}{2}V(x_j)$
of the $j$-th point
in that proof is replaced with $\frac{\beta}{2}V(x_j)+ Q^\tau_j(x_j)$.
The final estimate is somewhat weaker since now the bound on the constant $C_V$
defined in \eqref{CVdef} deteriorates to
$C_V\le C\tau^{-1}\le CN$. This accounts for the change from $N^{-2}$ to $N^{-3}$
in \eqref{tailprob}.
The argument for the measure $\mu_\by^{ B,\tau}$ is analogous and this
proves Lemma~\ref{lm:rep}.
\end{proof}

\subsection{Regularization does not change spacing statistics}

Given that the local relaxation measure $\mu_\by^\tau$ and its regularized
version $\mu^{ B,\tau}_\by$
are close in relative entropy sense, the next proposition shows that their
local spacing statistics coincide.

\begin{proposition}\label{prop:spacingB}
Let $\by\in \cG$, $\tau=\tau(\by)=\htau/s(\by)^2$ and assume
that for the parameters $B=N^{b}$, $K=N^k$ and $\htau = N^{-t}$
it holds
that
\be\label{cond2}
  1+ 2b -t-k< 0.
\ee
Then
\be\label{38B}
\Bigg| \big[ \E_{ \mu_{\by}^{\tau} } -\E_{ \mu_{\by}^{B, \tau}} \big]
 \frac 1 K \sum_{i \in I} G\Big( N(x_i-x_{i+1}) \Big)
\Bigg|
\to 0
\ee
as $N \to \infty$
for any smooth and compactly supported test function $G$.
\end{proposition}

\begin{proof}
Since Lemma~\ref{lm:SS} and \eqref{cond2} guarantee that
\be\label{22}
 \frac {N S( \mu_{\by}^{B, \tau} \mid  \mu_{\by}^{\tau})
 \tau }  K \le  \frac {C N B^2 \tau }  K     \log N  \le N^{-\e'}
\ee
with some $\e'>0$,
Proposition~\ref{prop:spacingB} is a direct consequence of
 the following comparison
lemma which was first stated in  a remark after
Lemma 3.4 in \cite{ESY4}, see also Lemma 4.4 in \cite{ESYY}.
\end{proof}

\begin{lemma}\label{lm:Scomparison}
Let  $G:\bR\to\bR$
 be a bounded smooth function with compact support and let
a sequence $E_i$ be fixed. Let $I$ be an interval of indices with
$|I|=K$. Consider a measure $\om$ with relaxation time
$\tau$ and let $q\rd \om$ be another probability measure.
Then for any $\e_1>0$ and for any smooth compactly supported function  we have
\be\label{Diff}
\Big|\frac 1 K \sum_{i \in I}   \int G\big( N(x_i-E_i )\big) [q-1]\rd \omega\Big|
\le C\sqrt { \frac{ N^{1+\e_1} S_\om( q) \tau}{K}}+  Ce^{-cN^{\e_1}}\sqrt {S_\om(q)}
\ee
and
\be\label{Diff1}
\Big|\frac 1 K \sum_{i \in I}   \int G\big( N(x_i-x_{i+1} )\big) [q-1]\rd \omega
\Big|
\le C\sqrt { \frac{ N^{1+\e_1} S_\om( q) \tau}{K}}+  Ce^{-cN^{\e_1}}\sqrt {S_\om(q)} ,
\ee
where $S_\om(q):=S(q\om\mid \om)$.
\end{lemma}

\begin{proof} Let $q$ evolve by the dynamics $\pt_t q_t = \cL q_t$, where
$\cL$ is the generator defined by
\be\label{def:dir}
     \int -f \cL f \rd \om = D_\om(f)=\frac{1}{2N}\int |\nabla f|^2 \rd \om.
\ee
Let $\tau_1 = N^{\e_1} \tau$.
Since $q_{\tau_1}$ is already subexponentially close to $\om$ in entropy sense,
$S_\om(q_{\tau_1})\le C\exp(-cN^{\e_1})S_\om(q)$,
and the total variation norm can be estimated by the relative entropy,
we only have to  compare $q$ with $q_{\tau_1}$.

By  differentiation, we have (the summation over $i$ always runs $i\in I$)
\begin{align}
\int \frac 1 K \sum_{i} & G\Big( N(x_i-E_i ) \Big)  q_{\tau_1} \rd \omega  -
\int \frac 1 K \sum_{i}  G\Big( N(x_i-E_i) \Big)  q \rd \omega \\
&= \int_0^{\tau_1}  \rd s \int  \frac 1 K \sum_{i}
   \pt_i G\Big( N(x_i-E_i))\Big)
\pt_{i} q_s   \rd \omega.
\end{align}
Here we used the definition of $ \cL$ from \eqref{def:dir} and note that
the $1/N$ factor present in \eqref{def:dir} cancels the factor $N$
from the argument of $G$.
 {F}rom the Schwarz inequality and $\pt q = 2 \sqrt{q}\pt\sqrt{q}$,
the last term is bounded by
\begin{align}\label{4.1}
 \Big[  \frac {N} { K^2} \int_0^{\tau_1}  \rd s \int &
\sum_{i}  \Big[\pt_i G \big(N(x_i -E_i ) \big)\Big] ^2
 \, q_s \rd \omega
\Big]^{1/2} \left [ \int_0^{\tau_1}  \rd s \int  \frac 1 {N } \sum_{i}
 (\pt_{i}\sqrt {q_s})^2  \rd \omega \right ]^{1/2} \nonumber \\
\le &  \; C \sqrt { \frac{ N S_\omega(q) \tau_1}{K}}
\end{align}
by integrating $\pt_s S_\om(q_s) =-4 D_\om(\sqrt{q_s})$.
This proves \eqref{Diff} and the proof of \eqref{Diff1} is analogous.
\end{proof}

\subsection{Accuracy of block averages}

In the next Section~\ref{sec:dir} we will compare the regularized local relaxation measure
$\mu_\by^{B,\tau}$ with the reference measure $\g_\th^{\htau,s(\by)}$
in Dirichlet form sense. As a preparation for this step, we give an estimate
on the location of the block averages $x_j^{[B]}$. Recall their definition
$$
   x_j^{[B]}:= \frac{1}{2B+1} \sum_{|k-j|\le B} x_k
$$
for any  $j\in \llbracket L+B+1,  L+K-B\rrbracket$.
The following lemma shows  concentration on a scale $\zeta$ for $x_j^{[B]}$
w.r.t. $\mu_\by^\tau$ and $\mu_\by^{B, \tau}$. The scale $\zeta$ is
larger than $\delta$ but will be smaller  than  $K/N$, the length of
configuration space interval.  Thus that the accuracy of the position
 of $x_j$  decreases from $\delta$ to $\zeta$, but
the accuracy of $y_k$ is still $\delta$.

\begin{lemma} \label{lm:accuracy}
Set $\zeta = N^{-z}$, $t=2d-1-\varphi$ and fix $\by\in \cG$.
For any $j\in \llbracket L+B+1,  L+K-B\rrbracket$ we have
\be
\label{cont1}
   \P_{\mu_\by^\tau}\Big( \big| x_j^{[B]}
  - \E_{\mu_\by^\tau}  x_j^{[B]}\big|\ge \zeta\Big)\le c_1 e^{-c_2 N^{\e'}}
\ee
and
\be
\label{cont2}
   \P_{\mu_\by^{B,\tau}}\Big( \big| x_j^{[B]}
  - \E_{\mu_\by^{B,\tau}}  x_j^{[B]}\big|\ge \zeta\Big)\le c_1 e^{-c_2 N^{\e'}}
\ee
provided
\be\label{66}
 z  \le -\varphi+ \min \Big( d- \frac{b}{2}-\frac{\varphi}{2},
\; d-\frac{k}{2}+ \frac{b}{2}\Big)
\ee
for some $\e'=\e'(d,\varphi)>0$ depending only on $d$ and $\varphi$.
Furthermore, we have
\be\label{E-E}
   \Big | \E_{\mu_\by^{B,\tau}} x_j^{[B]}  -    \gamma_j^{[B]}
  \Big | \le 5\zeta, \qquad
\Big | \E_{\mu_\by^\tau} x_j^{[B]}  - \gamma_j^{[B]}\Big|\le 5\zeta, \qquad
\Big | \E_{\mu_\by} x_j^{[B]}  - \gamma_j^{[B]}\Big|\le 5\zeta.
\ee
\end{lemma}

\begin{proof}
We will need two standard inequalities from probability theory.
The first one is
\be
\P_\mu (A) \cdot \log \frac{1}{\P_\nu(A)} \le  \log 2 + S(\mu| \nu)
\label{PPS}
\ee
for any set $A$ and probability measures $\mu, \nu$.
This can be obtained from the entropy inequality
$$
   \int f \rd \mu \le S(\mu| \nu) + \log \Big[ \int e^f \rd \nu\Big]
$$
by choosing $f(x) = b \cdot {\bf 1}_A(x)$ with $b= - \log \P^\nu(A)$.
Using Lemma~\ref{lm:SS} we thus obtain
\be\label{Acomp}
\P_{\mu_\by^{B,\tau}} (A)\le \frac {\log 2
+ CB^2\log N} { -\log \P_{\mu_\by^\tau} (A)}.
\ee

The second inequality is a concentration
estimate. Suppose that the
 probability measure $\om$
 satisfies the logarithmic Sobolev inequality (LSI), i.e.
\be
    S_\om(f) \le C_{\text{s}} \int |\nabla \sqrt{f}|^2 \rd \om
\label{lsi}
\ee
holds for any $f\ge0$ with $\int f\rd \om=1$. Then for any
random variable $X$ with $\E_\om X=0$ and any number $T>0$ we have
\be
  \E_\om e^{TX}\le \E_\om
\exp \left( \frac{C_{\text{s}} T^2}{2} \, |\nabla  X|^2 \right).
\label{conc}
\ee

Since the Hamiltonian $\cH_\by^\tau$ is convex with
$\nabla^2 \cH_\by^\tau\ge \tau^{-1}$,
by the Bakry-Em\'ery criterion the measure $\mu_\by^\tau\sim
\exp(-N \cH_\by^\tau)$
satisfies \eqref{lsi} with Sobolev constant $C_s= 2\tau/N$,
i.e.
\be\label{SleqD}
S(\nu\mid \mu_\by^\tau)\le  4\tau D(\nu\mid \mu_\by^\tau)
\ee
for any probability measure $\nu$ (recall that the
definition of the Dirichlet form \eqref{def:Dirform} contains
a $1/2N$ prefactor).
The same statements hold for the regularized measure
$\mu_\by^{B,\tau}$.

For $L+B+1\le j \le L+K-B$ define the event
\be
A= A_j = \big\{ \big|x_j^{[B]} -  \E_{\mu_\by^\tau} x_j^{[B]}  \big|
 \ge \zeta \big\},
\qquad \mbox{with}\quad \zeta  = N^{-z},
\ee
with a parameter $z\in (0,1)$  chosen later.
Using \eqref{conc} for $X=\pm(x_j^B -  \E_{\mu_\by^\tau} x_j^B)$
and noticing that $|\nabla X|^2 =(2B+1)^{-1}$, we obtain
\be\label{Atau}
\P_{\mu_\by^\tau} (A) \le 2e^{ - \frac{1}{2}N B \zeta^2 \tau^{-1} }.
\ee

Using now \eqref{Acomp}, we get
\be\label{BA}
\P_{\mu_\by^{B,\tau}} (A)\le \frac{ CB\tau}{N\zeta^2} \to 0
\ee
assuming
\be
 b -t + 2 z -1 < 0.
\ee
Using $t=2d-1-\varphi$,
we need
\be\label{662}
 z  < d-\frac{b}{2}-\frac{\varphi}{2}.
\ee
Under this condition we have from \eqref{Atau} that
\be
 \P_{\mu_\by^\tau} \Big( \big|x_j^{[B]} -  \E_{\mu_\by^\tau} x_j^{[B]}  \big|
 \ge \zeta \Big) \le 2e^{ - B^2 }.
\ee

Since the measure $\mu_\by^{B,\tau}$ is also concentrated by the LSI, we have
$$
   \P_{\mu_\by^{B,\tau}}\Big( \big| x_j^{[B]}-  \E_{\mu_\by^{B,\tau}}   x_j^{[B]}
  \big|\ge \zeta
  \Big)\le 2e^{ - B^2 }\to0
$$
and together with \eqref{BA} we have
\be\label{E-E1}
 \Big | \E_{\mu_\by^{B,\tau}} x_j^{[B]}  -    \E_{\mu_\by^\tau} x_j^{[B]}
  \Big | \le 2\zeta.
\ee
Therefore $x_j^{[B]}$ is concentrated on a scale $\zeta$ around the same
point w.r.t both measures $\mu_\by^\tau$ and $\mu_\by^{B,\tau}$.

Using \eqref{SleqD} and that $\by\in \cG$ we get
\be\label{332}
S( \mu_\by | \mu_\by^\tau ) \le
 4\tau  D( \mu_\by | \mu_\by^\tau)
  \le  \frac {4N}{ \tau}   \E_{ \mu_\by}    \sum_{j \in I} (x_j - \gamma_j)^2
\le  \frac { 4N \delta^2 K} \tau.
\ee
Hence by \eqref{PPS} and \eqref{Atau} we obtain
\be\label{PPPA}
\P_{\mu_\by} (A)\le \frac {\log 2 +
 \frac { 4N \delta^2 K} \tau  } {- \log \P_{\mu_\by^\tau} (A) }
 \le  \frac { C\delta^2 K }{ B\zeta^2}
\to 0
\ee
provided that
\be
  z < d-\frac{k}{2}+\frac{b}{2}.
\ee

Now by the definition of $\by\in \cG$ in \eqref{goodset} we have
\begin{align}
    \P_{\mu_\by}\Big( \big|x_j^{[B]}&
- \E_{\mu_\by} x_j^{[B]}\big|\ge \zeta\Big) \nonumber
 \le \zeta^{-2}
\E_{\mu_\by}\big|x_j^{[B]} - \E_{\mu_\by} x_j^{[B]}\big|^2 \\
& \le\frac{1}{(2B+1)\zeta^2} \sum_{|k-j|\le B}
\E_{\mu_\by}\big|x_k - \E_{\mu_\by} x_k\big|^2 \nonumber \\
& \le\frac{1}{(2B+1)\zeta^2} \sum_{|k-j|\le B}
\E_{\mu_\by}\big|x_k - \gamma_k\big|^2
\le\frac{\delta^2}{\zeta^2}\to 0  \nonumber
\end{align}
using \eqref{662}. Combining it with \eqref{PPPA} we
obtain
\be\label{E-E2}
\Big | \E_{\mu_\by} x_j^{[B]}  -    \E_{\mu_\by^\tau} x_j^{[B]}
  \Big | \le 2\zeta.
\ee
Finally, since $\by\in\cG$, we have
$$
   \Big( \E_{\mu_\by} x_j^{[B]}- \gamma_j^{[B]}\Big)^2 \le
   \E_{\mu_\by} \Big( x_j^{[B]}-  \gamma_j^{[B]}\Big)^2 \le \delta^2\le\zeta^2,
$$
which, combined with \eqref{E-E1} and \eqref{E-E2}, yields \eqref{E-E}.
This completes the proof of Lemma~\ref{lm:accuracy}.
\end{proof}

\subsection{Proof of Proposition~\ref{prop:musig}}\label{sec:musig}

Now we will compare the regularized local relaxation measure
$\mu_\by^{B,\tau}$
 with the reference measure $\g_\th^{\htau,s}$
in Dirichlet form sense. Recall their definitions from
\eqref{def:regloc} and \eqref{def:ref}, respectively,
and recall that $\tau=\tau(\by):= \htau/s(\by)^2$.
Here $s=s(\by)$ is a function that is approximately 1 for
good external configurations $\by\in \cG$ (see \eqref{s-1}).

The  result is the following comparison of local gap statistics.
Combining this result with
Proposition~\ref{prop:spacingB} and checking that the condition
\eqref{cond2} is satisfied with the choice of parameters given below,
we arrive at the proof of Proposition~\ref{prop:musig}. \qed

\begin{proposition}\label{prop:spacingD} Fix $\varphi\le \frac{1}{38}$.
Let $\by\in \cG$, $\tau=\tau(\by)=\htau/s(\by)^2$ and assume
that for the parameters $\delta=N^{-d}$, $B=N^{b}$, $K=N^k$
with $d=1-\varphi$, $b=8\varphi$, $k=\frac{39}{2}\varphi$. Then with $t=2d-1-\varphi=1-3\varphi$
let  $\htau = N^{-t}$ with $t:=2d-1-\varphi=1-3\varphi$.
Then
\be\label{38D}
\Bigg| \big[ \E_{ \mu_{\by}^{B, \tau} } -\E_{ \g_{\by}^{\htau,s}} \big]
 \frac 1 K \sum_{i \in I} G\Big( N(x_i-x_{i+1}) \Big)
\Bigg|
\to 0
\ee
as $N \to \infty$
for any smooth and compactly supported test function $G$.
\end{proposition}

\begin{proof}
The key technical estimate is the following lemma whose proof will
take up most of this section.

\begin{lemma}\label{eb} Let $\varphi>0$. Suppose $B=N^b$, $K=N^k$ with $0<b<k<1$, and $\delta= N^{-d}$
with  $d\in (0,1)$. Suppose that these parameters satisfy
\be
   1-b <  -\varphi+ \min \Big( d- \frac{b}{2}-\frac{\varphi}{2},
\; d-\frac{k}{2}+ \frac{b}{2}\Big),
\label{p1}
\ee
i.e. one can choose a number $z>1-b$ and satisfying \eqref{66}.
Let $\by\in\cG=\cG_{\delta, \e_0}$
be a good configuration. Assume that $\e_0\le \e'/10$, where $\e'=\e'(d,\varphi)$
is obtained in Lemma~\ref{lm:accuracy}.
Assume that the equilibrium measure $\rho_V$ is $C^1$ away from the edges.
Let  $\htau = N^{-t}$
with $t=2d-1-\varphi$.
Then the Dirichlet form of $ \mu_{\by}^{B,\tau}$ with respect to the reference measure
 is bounded by
\be\label{largebound}
  \frac{\tau}{K}  D \big(\mu_\by^{B,\tau}\mid \g_\th^{\htau,s})
  \le  C\htau (\log N)  \Big [ 
 \frac {K^2 } {N} +
 \frac {N  \delta^2}  { \htau^2}+   \frac {K^4} {N^3\htau^2}
  +  \frac { \delta^2 N^3} { BK} +   N^{3/5 + \varphi} \Big ]
 + c_1 e^{-c_2N^{\e'/3}}.
\ee
\end{lemma}
The prefactor $\tau/K$ is for convenience; the local
gap statistics of two measures are approximately the same
if $\tau D/K\to 0$. More precisely, we have the following
general theorem which is a slight modification of
Lemma 3.4 \cite{ESY4} (see also Theorem~4.3 in \cite{ESYY}).
This result was originally proven for $\beta\ge 1$,
but by a regularization argument
it  extends to any $\beta>0$, see  Lemma A.2 of \cite{EKYY2}
for details.

\begin{lemma}\label{lm:Dcomparison}
Let  $G:\bR\to\bR$
 be a bounded smooth function with compact support.
 Consider a measure on $\Sigma_K:=\{ \bx\; :\;  x_1<   \ldots < x_K\} \subset \RR^K$ defined by 
\begin{equation}\label{01}
\rd \om\sim e^{- \beta N \wh \cH}\rd\bx,
\quad
\wh \cH(\bx) =  \cH_0(\bx)
-\frac{1}{N} \sum_{1\leq i<j\leq K}\log (x_j-x_i) ,
\end{equation}
with the property that  $\nabla^2 \cH_0 \ge \tau^{-1}$ holds for some positive constant $\tau$.  
Let $q\rd \om$ be another probability measure.
 Let $I\subset\{1, 2, \ldots, K-1\}$ be an interval of indices.
 Then for any $\e_1>0$ and for any smooth compactly supported function  we have
\be\label{Diff1new}
\Big|\frac{1}{|I|} \sum_{i \in I}   \int G\big( N(x_i-x_{i+1} )\big) [q-1]\rd \omega
\Big|
\le C\sqrt { \frac{ N^{\e_1} D_\om( \sqrt q) \tau}{|I|}}+  Ce^{-cN^{\e_1}}\sqrt {S_\om(q)} ,
\ee
where $D_\om(\sqrt q):=D(q\om\mid \om)$.
\end{lemma}

We will apply Lemma~\ref{lm:Dcomparison}
for the measure  $\om =\g_\th^{\htau,s}$. 
It has the form \eqref{01} except that
$\g_\th^{\htau,s}$ is restricted to the interval $[y_L, y_{L+K+1}]$, i.e.
the relations $y_L< x_{L+1}$ and $x_{L+K}< y_{L+K+1}$ also hold in
addition to the ordering relation $ x_{L+1} < x_{L+2} < \ldots
 < x_{L+K}$.  Notice that the Hamiltonian   
$\cH_{\th}^{\htau, s}$ of the measure $\g_\th^{\htau,s}$
contains a term  $\frac{1}{N} 
\big[ \log (x_{L+1} - y_L ) + \log (y_{L+K+1}-x_{L+K})\big] $. 
This term
 confines the particles in the interval $[y_L, y_{L+K+1}]$
exactly as the term $\log (x_{i+1}-x_{i})$ guarantees 
the ordering constraint $x_i< x_{i+1}$.   Hence the regularization argument 
in Lemma A.2 of \cite{EKYY2}  can be used to treat  the additional constraints, $y_L < x_{L+1}$
and $x_{L+K}< y_{L+K+1}$.

The proof of Proposition~\ref{prop:spacingD} now follows
from Lemma~\ref{eb} and Lemma~\ref{lm:Dcomparison}
with  $\om =\g_\th^{\htau,s}$ and $q\rd\om = \mu_\by^{B,\tau}$.
The parameters $b,k,d\in (0,1)$ have to satisfy the following relations
from \eqref{p1} and from the requirement that
the right side of \eqref{largebound} converges to zero:
\begin{align}
b &< k \non \\
1-b+\varphi&< d-\frac{b}{2}-\frac{\varphi}{2} \non \\
1-b+\varphi&<  d-\frac{k}{2}+\frac{b}{2} \non\\
1-2d+\varphi +2k -1&< 0\non\\
1-2d + (2d-1-\varphi)&< 0\non\\
4k-3+(2d-1-\varphi)&< 0\non\\
1-2d+\varphi -2d+3-b-k&< 0\non \\
 1-2 d + 2 \varphi + \frac 3 5 &< 0. \non
\end{align}
It is easy to check that all these conditions are satisfied
if, e.g.
$$
   d=1-\varphi, \qquad  b=8\varphi, \qquad k=\frac{39}{2}\varphi, \qquad 0<\varphi \le \frac{1}{38}.
$$
This choice is not optimal for the above system of inequalities,
but we took into account that the parameters will also have
to satisfy \eqref{cond2} so that we could combine
Proposition~\ref{prop:spacingD} and Proposition~\ref{prop:spacingB}
to arrive at Proposition~\ref{prop:musig}.

Finally, the entropy term $S \big(\mu_\by^{B,\tau}\mid \g_\th^{\htau,s})$
in \eqref{Diff1new} can be estimated by the Dirichlet form via
the logarithmic Sobolev inequality.
This completes the proof of Proposition~\ref{prop:spacingD}.
\end{proof}

\bigskip

\subsection{Dirichlet form estimate: proof of Lemma~\ref{eb}}\label{sec:dir}

By definition,
\[  \frac{\tau}{K}  D \big(\mu_\by^{B,\tau}\mid \g_\th^{\htau,s})
= \frac{\tau }{2N K }\int  \Big|\nabla \log \Big( \frac{\mu_\by^{B,\tau}}{\g_\th^{\htau,s}}
\Big)\Big|^2 \rd   \mu_{\by}^{B,\tau} \le
\frac{\tau N}{ K} \int  \sum_{ L+1\le  j \le L+ K }  Z_j^2 \rd  \mu_{\by}^{B,\tau},
\]
where $Z_j$ is defined as follows:
For $L+1 < j \le L + 4B$, we set
\begin{align}
Z_j := &
 \frac{\beta}{2}V'(x_j) -
  \frac \beta N \sum_{    k < L-2B\atop k >L+K }
   \frac 1 {x_j- y_{k}^B}
 - \frac{\beta}{2} W_s '( x_j) +
 \frac \beta N \sum_{  k <  L-2B\atop k>L+K}   \frac 1 {x_j-  \th_{k}'}
   +     \frac { \gamma_{j}  - \theta_{j}' }  {  \tau  } \nonumber
\end{align}
(recall that $\th'_j=\th_j/s$ and we set $W_s(x)=s^2x^2$).
Note that the summation at the lower edge is only for $k< L-2B$ instead of $k\le L$
because  the interaction terms near the
boundary cancel by \eqref{ybth}. Moreover, notice that the linear terms,
coming from the derivative of the quadratic confinements (see \eqref{232}
and \eqref{242}),  cancel each other
$$
  \frac{s(\by)^2}{\htau}(x_j- \th_j')- \frac{1}{\tau} (x_j- \gamma_j)
  =  \frac { \gamma_{j}  - \theta_{j}' }  {  \tau  }
$$
by the choice of $\tau(\by) = \htau/s(\by)^2$.

Similarly, for  $L+K-4B < j \le L + K$, we set
\begin{align}
Z_j := &
 \frac{\beta}{2}V'(x_j) -
  \frac \beta N \sum_{    k > L+K+2B \atop k<L }
   \frac 1 {x_j- y_{k}^B}
 - \frac{\beta}{2} W_s '( x_j) +  \frac \beta N \sum_{  k >  L+K+2B\atop k<L}
  \frac 1 {x_j-  \th_{k}'}
   +     \frac { \gamma_{j}  - \theta_{j}' }  {  \tau  }. \nonumber
\end{align}
Finally,
for $L +4B < j \le L +K- 4B$, we define
\[
Z_j: =
\frac{\beta}{2} V'(x_j) -   \frac \beta N \sum_{    k < L \atop k>L+K+1}   \frac 1 {x_j- y_{k}}
 - \frac{\beta}{2} W_s '( x_j) +  \frac \beta N \sum_{  k <  L\atop k>L+K+1}
  \frac 1 {x_j-  \th_{k}'}
   +     \frac { \gamma_{j}  - \theta_{j}' }  {  \tau  }.
\]
Notice that here $y_k$ is not replaced with $y_k^B$  since
only interactions for $x_j$'s  near the edges have been regularized.
Moreover, the interactions with the boundary terms $y_k$, with $k=L$ and
$k=L+K+1$ cancel out since $y_L=\th_L'$ and $y_{L+K+1}=\th'_{L+K+1}$ by
the matching construction.

\bigskip

Now we estimate the size of $Z_j$ in each case.

\bigskip

\underline{\it Case 1: $L +4B < j \le L +K- 4B$}.
The first step is to decompose $Z_j$ as
\be
Z_j
= \beta \sum_{a=1}^5  \Omega_j^a,
\ee
where
\begin{align}
\Omega_j^1 &
 : =  \left [  \frac{1}{2} V'(x_j) -\int \rd y  \frac {\rho_V(y) } {x_j- y }  \right ]
- \left [  \frac{1}{2} W_s '( x_j) -  \int \rd y  \frac {\rho_{W_s}( y) } { x_j- y }  \right ]
\nonumber\\
\Omega^2_j & = \Omega^{2,low}_j+\Omega^{2,up}_j \non\\
& : = - \Bigg( \frac 1 N  \sum_{  k< L   }
 \frac 1 {x_j- y_{k}} - \int_{ -\infty}^{y_{L}}    \frac {\rho_V(y) } {x_j- y }  \rd y\Bigg)
  - \Bigg( \frac 1 N  \sum_{  k > L+K+1   }
 \frac 1 {x_j- y_{k}} - \int_{y_{L+K+1}}^{\infty}    \frac {\rho_V(y) } {x_j- y }  \rd y\Bigg)
\nonumber\\
\Omega^3_j &  = \Omega^{3,low}_j+\Omega^{3,up}_j \non \\
& := \Bigg( \frac 1 N  \sum_{  k<   L }   \frac 1 {x_j-  \th_{k}'} -
    \int_{ -\infty}^{ \theta_{ L}'}    \frac {\rho_{W_s}(y) } { x_j- y }  \rd y\Bigg)
  + \Bigg( \frac 1 N  \sum_{  k>  L+K+1 }   \frac 1 {x_j-  \th_{k}'} -
    \int_{ \th_{L+K+1}'}^{ \infty}    \frac {\rho_{W_s}(y) } { x_j- y }  \rd y\Bigg)
\nonumber\\
\Omega^4_j & : = \int_{y_{L}}^{ y_{L+K+1} }    \frac {\rho_V(y) - \rho_{W_s}(y)  } {x_j- y } \rd y
\nonumber\\
\Omega^5_j & : =    \frac { \gamma_{j}  - \theta_{j}' }  {  \beta \tau  }.
\label{fiveom}
\end{align}
Here we also used that $[y_L, y_{L+K+1}]=[\th_L', \th_{L+K+1}']$ when
establishing the limits of integrations.
By the equilibrium relation \eqref{equilibrium}
between $V$ and $\rho_V$, we have
\be \Omega^1_j=0.
\label{om1}
\ee

{F}rom \eqref{38}, we have
\be\label{Om2}
[\Omega^5_j]^2 = C\frac {( \gamma_{j}- \theta_{j}')^2} {\tau^2}
 \le \frac C {\tau^2} \left [ \delta^2+  \frac { K^4 } { N^4 } \right ].
\ee

Since  $\rho_V\in C^1$  away from the edge, and so is the semicircle
density $\rho_{W_s}$, we have by Taylor expansion
\begin{align}
|\Omega^4_j| & = \Big|\int_{y_{L}}^{ y_{L+K+1} }
    \frac {\rho_V(y) - \rho_{W_s}(y)  } {x_j- y } \rd y\Big|
\label{Om4}\\
& \le
\Bigg| \int_{y_{L}}^{ y_{L+K+1} }
\frac {\rho_V(x_j)  - \rho_{W_s}(x_j)    + O( x_j-y )  } {x_j- y }  \rd y
\Bigg|  \nonumber\\
& \le C  \big[ |\log (x_j-y_L)|+|\log (y_{L+K+1}-x_j)|\big]   \left [
  \frac K N + \frac { \delta N } K \right ]. \nonumber
\end{align}
Here we used \eqref{ap} and \eqref{VW0} and the fact that
$\rho_{W_s}(x)-\rho_W(x) = O(|s-1|)$ away from the edge
together with \eqref{s-1} to estimate
$$
 |\rho_V(x)  - \rho_{W_s}(x)|\le C\left[ \frac K N + \frac { \delta N } K \right]
$$
for any $x\in [y_L, y_{L+K+1}]$.
The logarithmic terms after taking square and expectation w.r.t.
 will give rise to an irrelevant $\log N$ factor by
using Lemma~\ref{lm:rep}
$$
  \E_{\mu_\by^{B,\tau}} \big[ |\log (x_j-y_L)|+|\log (y_{L+K+1}-x_j)|\big]^2\le C\log N.
$$
\medskip

We now estimate the main error  $\Omega^2_j $ and we will deal with
the first term only, coming from the lower edge, the second one can be treated similarly.
We write it as
$$
\Omega_j^{2,low}=-\Bigg(  \frac 1 N  \sum_{  k< L   }
 \frac 1 {x_j- y_{k}} - \int_{ -\infty}^{y_{L}}    \frac {\rho_V(y) } {x_j- y }  \rd y\Bigg)
=\Omega_j^{2,1}+\Omega_j^{2,2} + \Omega_j^{2,3}
$$
with
\begin{align}
\Omega_j^{2,1} & :=  -\Bigg( \frac 1 N  \sum_{  k< L   }
  \frac 1 {x_j- \gamma_{k}} - \int_{ -\infty}^{\gamma_{L}}    \frac {\rho_V(y) } {x_j- y }  \rd y\Bigg)
\non\\
\Omega_j^{2,2} & :=  \int_{\gamma_{L}}^{y_L}    \frac {\rho_V(y) } {x_j- y }  \rd y
\non\\
\Omega_j^{2,3} & :=  \frac 1 N  \sum_{  k< L   }
  \Big[ \frac 1 {x_j- \gamma_{k}}-  \frac 1 {x_j- y_{k}} \Big].
\label{Om3decomp}
\end{align}
With $\zeta= N^{-z}$ with $z$ is given in Lemma~\ref{eb},
define the event
$$
\Lambda=\Big\{ |x_i^{[B]} -\gamma_i^{[B]}|\le 6\zeta, \quad
\forall i \in\llbracket L+B+1, L+K-B\rrbracket\Big\},
$$
then its complement has very small probability,
$$
  \P_{\mu_\by^{B,\tau}}\big( \Lambda^c\big)\le c_1 e^{-c_2N^{\e'}}
$$
from  \eqref{cont1} and \eqref{E-E}.
On the event $\Lambda^c$ we simply estimate
$$
  \Big| \frac 1 N  \sum_{  k< L   }
 \frac 1 {x_j- y_{k}} - \int_{ -\infty}^{y_{L}}    \frac {\rho_V(y) } {x_j- y }  \rd y\Big|
\le \frac{1}{y_L-y_{L-1}}+ C\big|\log (x_j-y_L)\big| ,
$$
therefore
\begin{align}
  \E_{\mu_\by^{B,\tau}} {\bf 1}(\Lambda^c) & \Big| \frac 1 N  \sum_{  k< L   }
 \frac 1 {x_j- y_{k}} - \int_{ -\infty}^{y_{L}}    \frac {\rho_V(y) } {x_j- y }  \rd y\Big|^2
\nonumber\\
  & \le C\Big(\frac{1}{(y_L-y_{L-1})^2}+  \E_{\mu_\by^{B,\tau}}\big|\log (x_j-y_L)\big|^4\Big)^{1/2}
 \big( \P_{\mu_\by^{B,\tau}}\big( \Lambda^c\big)\big)^{1/2}\le c_1 e^{-c_2N^{\e'/3}}
\label{subb}
\end{align}
by using Lemma~\ref{lm:rep} and $|y_L-y_{L-1}|\ge \exp(- N^{\e_0})$ from
$\by\in \cG$. Here we used that $\e_0\le \e'/10$.

\medskip

Now we continue the estimate on the set $\Lambda$
and we consider the three terms in \eqref{Om3decomp} separately.
For the first term  we write
$$
   \Omega_j^{2,1} =  \frac 1 N  \sum_{  k< L   } \Big[
  \frac 1 {x_j- \gamma_{k}} -
\int_{ \gamma_{k}}^{\gamma_{k+1}}    \frac {N\rho_V(y) } {x_j- y }  \rd y\Big]
=  \frac{1}{N} \sum_{  k< L   }
\int_{\gamma_k}^{\gamma_{k+1}}  \Gamma_{j}^k   N\rho_V(y) \rd y
$$
where we have used that $\int_{\gamma_k}^{\gamma_{k+1}} N\rho_V=1$ and
$$
   \Gamma_{j}^k  =       \frac {\gamma_k-y } {(x_j- y)(x_j-\gamma_k) }.
$$
Recall that  $L\ge \kappa N\gg \delta$ and
$x_j\in [y_L, y_{L+K+1}] = [\gamma_L,\gamma_{L+K+1}] + O(\delta)$.
For $k\le \frac{1}{2}\kappa N$ we know that $|\gamma_k-x_j|\ge c$ with some positive constant. Hence we have
$$
  \frac{1}{N}\sum_{k\le \kappa N/2} \Gamma_{j}^k  \le
  \frac C N  \sum_{  k\le \kappa N/2   }
\int_{\gamma_k}^{\gamma_{k+1}}    |\gamma_k-y|  N\rho_V(y) \rd y
\le  \frac C N  \sum_{  k\le \kappa N/2   }   |\gamma_{k+1}-\gamma_k|\le CN^{-1},
$$
since $\gamma_{k+1}-\gamma_k \le CN^{-2/3}k^{-1/3}$ near a square root
singularity of $\rho_V$ at the edge. For the regime $k\ge \frac{1}{2}\kappa N$ we can use
$|\gamma_{k+1}-\gamma_k|\le CN^{-1}$ to get
\begin{align}
  \frac{1}{N}\sum_{ \kappa N/2\le k< L}  \Gamma_{j}^k  & \le
\frac{1}{N}\sum_{ \kappa N/2\le k< L} \frac{C}{N} \frac{1}{(x_j-\gamma_k)^2} \nonumber \\
& \le \frac{1}{N}\sum_{ \kappa N/2\le k< L} \frac{C}{N} \frac{1}{(x_{j-B}^{[B]}-\gamma_k)^2}
\nonumber \\
& \le \frac{C}{N} \frac{1}{(x_{j-B}^{[B]}-\gamma_L)}\le \frac{C}{B}.  \nonumber
\end{align}
Here in the second inequality we used that
 on the set $\Lambda$ we have
\be\label{xy}
x_j\ge x_{j-B}^{[B]}> \gamma_{j-B}^{[B]}-6\zeta
\ge \gamma_{j-2B}-6\zeta\ge \gamma_L + cBN^{-1}> \gamma_k+cBN^{-1}
\ee
 for $k<L$
using $j\ge L+4B$ and thus $\gamma_{j-2B}-\gamma_L\ge cBN^{-1}\gg 6\zeta$,
since $z>1-b$. Therefore $x_j-\gamma_k\ge x_{j-B}^{[B]}-\gamma_k>0$.
In the third inequality we performed the summation and used that $\gamma_k$
is regularly spaced. In the last inequality we again used \eqref{xy}.
In summary, we have shown that
\be\label{om21}
  |\Om_j^{2,1}|\le \frac{C}{B}+\frac{C}{N}
\ee
on the set $\Lambda$ and we have seen that the contribution
from $\Lambda^c$ is subexponentially small \eqref{subb}.

\medskip

Now we consider $\Om_j^{2,2}$ on $\Lambda$.
We have
\be\label{om22}
  |\Om_j^{2,2}|\le C \int_{y_L}^{\gamma_L} \frac{\rd y}{x_j-y}\le
\frac{C\delta}{\gamma_{j-2B}-\gamma_L},
\ee
by using $x_j-\gamma_L \ge \gamma_{j-2B}-\gamma_L -6\zeta
 \ge c( \gamma_{j-2B}-\gamma_L) $ from \eqref{xy}
and from  $\gamma_{j-2B}-\gamma_{L}\ge cBN^{-1}\gg 6\zeta$, moreover
$x_j-y_L\ge x_j-\gamma_L - \delta \ge c( \gamma_{j-2B}-\gamma_L)$
 by
$|\gamma_L-y_L|\le \delta$  (from $\by\in \cG$)
and $\delta\ll BN^{-1}$ (from \eqref{p1}).
Thus
$$
\sum_{L+4B\le j\le L+K-4B}  |\Om_j^{2,2}|^2
\le \frac{C\delta^2N^2}{B}.
$$

\medskip

For the third term $\Om_j^{2,3}$ we have
 \begin{align}
 \E_{\mu_\by^{B,\tau} }  {\bf 1}(\Lambda)
 & \sum_{L+4B < j \le L+ K-4B}  [\Omega^{2,3}_j ]^2  \label{om333} \\
& \le   \E_{\mu_\by^{B,\tau} }   {\bf 1}(\Lambda)
\sum_{L+4B < j \le L + K-4B}
\left [  \frac 1 N  \sum_{k < L } \Big(  \frac 1 {x_j- y_k} -     \frac 1 {x_j- \gamma_k}
 \Big)\right ] ^2
\nonumber \\
 & \le     \E_{\mu_\by^{B,\tau} }  {\bf 1}(\Lambda) \sum_{L+4B < j \le L + K-4B}
  \left [ \frac 1 N  \sum_{  k < L  }
 \frac {(y_k-\gamma_k) } { (x_j- y_k ) (x_j- \gamma_k )}  \right ] ^2.
\non
\end{align}
We split the summation over $k$ into two terms: $\kappa N/2\le k <L$ and $k<\kappa N/2$
and separate by a Schwarz inequality.

First we consider the case  $\kappa N/2 \le k < L$.
Expanding the square, we need to bound
\begin{align}\label{expsq}
\E_{\mu_\by^{B,\tau} }&  {\bf 1}(\Lambda)   \frac {1} {N^2}   \sum_{ \kappa N/2\le k < L }
 \sum_{\kappa N/2\le a< L}
 \sum_{L +4B < j   \le L+K }
     \frac { |y_k-\gamma_k|   |y_a-\gamma_a| } { (x_j- y_k) (x_j-\gamma_k)   (x_j- y_a)
   (x_j- \gamma_a)}     \\
&  \le   2 \E_{\mu_\by^{B,\tau} }  {\bf 1}(\Lambda)
  \frac 1 {N^2}     \sum_{ \kappa N/2\le k < L }  |y_k-\gamma_{k}|^2
    \sum_{L+4B \le j  \le L +  K }        \frac { 1   } { (x_j- y_{k})^2  }
 \sum_{\kappa N/2\le a < L}   \frac { 1   } {   (x_j- \gamma_a)^2} ,\non
\end{align}
where we used another Schwarz inequality and
the factor 2 accounts for a similar term with the role of $k$ and $a$
interchanged.

In the case  $\kappa N/2 \le k < L$ we have  $|\gamma_k-y_k|\le \delta$.
Then  \eqref{expsq} is bounded by
\begin{align}\label{frr}
 \frac {2\delta^2}  {N^2}  \sum_{L+4B \le j  \le L +  K } &
      \E_{\mu_\by^{B,\tau} }      {\bf 1}(\Lambda)   \sum_{ k < L } \frac { 1   } { (x_j- y_{k})^2  }
 \sum_ {a <L}   \frac { 1   } {   (x_j- \gamma_{a})^2}
\\
& \le  C \delta^2    \sum_{L+4B \le j  \le L +  K }
   \E_{\mu_\by^{B,\tau} }      {\bf 1}(\Lambda)   \frac { 1   } { (x_j- y_{L})^2  }
\le C\delta^2   \frac { N^2} { B}.   \non
\end{align}
Here we used
$$
  \frac{1}{N} \sum_ {a <L}   \frac { 1   } {   (x_j- \gamma_{a})^2} \le \frac{C}{x_j-\gamma_L}
  \le  \frac{C}{x_j-y_L}
$$
relying on the regularity of $\gamma_a$ and using, from \eqref{xy},
 that $x_j-\gamma_a\ge x_j-\gamma_L
\ge cBN^{-1}$ which is much larger than the spacing of order $N^{-1}$ of the $\gamma$-sequence.
In the last estimate $x_j-\gamma_L\gg |\gamma_L-y_L|$ was used (since $BN^{-1}\gg \delta$).
Similarly we could perform the $k$ summation
$$
 \sum_{ k < L } \frac { 1   } { (x_j- y_{k})^2  } \le  \frac{C}{x_j-y_L}
$$
since $x_j-y_k\ge x_j-\gamma_k -\delta  \ge c(x_j-\gamma_k)$.

To perform the $j$ summation in \eqref{frr}, we use
$$
 \frac { 1   } { (x_j- y_{L})^2  }
 \le  \frac { 1   } { (x_{j-B}^{[B]}- y_{L})^2  },
$$
and then we recall that apart from a set of subexponentially small
probability, we have
$$
   |x_{j-B}^{[B]}-\gamma_{j-B}^{[B]}|\le 6\zeta
$$
from Lemma~\ref{lm:accuracy}. Since $\zeta\ll BN^{-1}$ and
$x_{j-B}^{[B]}- y_{L}\ge cBN^{-1}$ from \eqref{xy},
we see that
$$
\sum_{L+4B \le j  \le L +  K }    \frac { 1   } { (x_j- y_{L})^2  }
\le\sum_{L+4B \le j }\frac { C   } { (\gamma_{j-B}^{[B]}- y_{L})^2  }
\le \frac{CN}{\gamma_{L+3B}^{[B]}- y_{L}}\le \frac{CN^2}{B}.
$$
On the exceptional set one can just use the trivial bound
$(x_j-y_L)^{-2}\le C(x_j-\gamma_L)^{-2}\le CN^2B^{-2}$ from \eqref{xy}.

\bigskip
Consider now  the case  $ k \le \kappa N/2$ in \eqref{om333}. We have
\begin{align} \label{6141}
 \E_{\mu_\by^{B,\tau }}   {\bf 1}(\Lambda)  \frac 1 {N^2} &
  \sum_{ k  \le \kappa N/2  }  \sum_{a \le \kappa N/2 }
 \sum_{L +4B \le j   \le L+K }
    \frac { |y_k-\gamma_{k}|   |y_a-\gamma_{a}|   } { (x_j- y_{k}) (x_j-\gamma_{k})   (x_j- y_{a})
   (x_j- \gamma_{a})}   \\
 &  \le   2\E_{\mu_\by^{B,\tau} }    {\bf 1}(\Lambda) \frac 1 {N^2}
   \sum_{ k  \le \kappa N/2  }  |y_k-\gamma_{k}|^2
  \sum_{L+4B \le j  \le L +  K }        \frac { 1   } { (x_j- y_{k})^2  }
 \sum_ {a \le \kappa N/2 }   \frac { 1   } {   (x_j- \gamma_{a})^2}\non \\
 &  \le   \frac {CK} {N}  \E_{\mu_\by^{B,\tau} }   {\bf 1}(\Lambda)
  \sum_{ k  \le \kappa N/2  }  |y_k-\gamma_{k}|^2    \le C KN^{-2/5+ \varphi},  \non
\end{align}
where we used that all denominators are separated away from zero
and Lemma \ref{edge}.
Furthermore, in the last inequality, we have used Lemma \ref{edge}
 for $k \ge N^{3/5 + \varphi}$ and
 we used $ |y_k-\gamma_{k}| \le O(1)$ for $k \le N^{3/5 + \varphi}$ from
 $\by\in \cG$ and \eqref{goodset}.
Similar comment applies to all edge terms in this proof and we will not repeat it.

Summarizing, we have shown that
\be\label{Om23}
\E_{\mu_\by^{B,\tau} }  {\bf 1}(\Lambda)
 \sum_{L+4B < j \le L+ K-4B}  [\Omega^{2,3}_j ]^2\le \frac{C\delta^2N^2}{B}+  C KN^{-2/5+ \varphi}.
\ee

Finally, we need to estimate $\Om_j^3$ in \eqref{fiveom}. It can
be treated exactly as $\Om_j^{2,1}$ and the result is
\be
\label{om3}
 |\Om_j^{3}|\le \frac{C}{B}+\frac{C}{N}
\ee
on the set $\Lambda$ and the contribution
from $\Lambda^c$ is subexponentially small as in \eqref{subb}.

\bigskip

\underline{\it Case 2: $L  < j \le L +4B$.}
 (There is a third case $L+K-4B\le j\le L+K$ which
is identical to Case 2 and will not be treated separately).
We  decompose $Z_j$ as before and the only modifications are
\begin{align}
\Omega^{2,low}_j&  := -\Bigg(  \frac 1 N  \sum_{  k\le L-2B   }
\frac 1 {x_j- y_{k}^B} - \int_{ -\infty}^{y_{L-2B}}   \frac {\rho_V(y) } {x_j- y }  \rd y\Bigg)
\non\\
\Omega^{3,low}_j& :=  \frac 1 N  \sum_{  k\le   L-2B }   \frac 1 {x_j-  \th_{k}'} -
 \int_{ -\infty}^{ \theta_{ L-2B}'}    \frac {\rho_{W_s}(y) } { x_j- y }  \rd y
\non\\
\Omega^4_j & := \int_{y_{L-2B}}^{ y_{L+K+1} }    \frac {\rho_V(y) - \rho_{W_s}(y)  } {x_j- y } \rd y
+\int_{y_{L-2B}}^{ \th'_{l-2B} }    \frac {\rho_{W_s}(y)  } {x_j- y } \rd y.
\non
\end{align}

We now estimate the main error term $\Omega^{2,low}_j $, and we write it, as before
$$
\Omega_j^{2,low}=\Omega_j^{2,1}+\Omega_j^{2,2} + \Omega_j^{2,3}
$$
with
\begin{align}
\Omega_j^{2,1} & :=  -\Bigg( \frac 1 N  \sum_{  k< L-2B   }
  \frac 1 {x_j- \gamma_{k}} - \int_{ -\infty}^{\gamma_{L-2B}}
  \frac {\rho_V(y) } {x_j- y }  \rd y\Bigg)
\non\\
\Omega_j^{2,2} & :=  \int_{\gamma_{L-2B}}^{y_{L-2B}}    \frac {\rho_V(y) } {x_j- y }  \rd y
\non\\
\Omega_j^{2,3} & :=  \frac 1 N  \sum_{  k< L-2B   }
  \Big[ \frac 1 {x_j- \gamma_{k}}-  \frac 1 {x_j- y_{k}} \Big].
\label{Om3newdecomp}
\end{align}
We have
\begin{align}
 |\Omega_j^{2,1}| & =\Bigg| \frac 1 N  \sum_{  k< L-2B   }
   \int_{ \gamma_k}^{\gamma_{k+1}}
  \frac {y-\gamma_k} {(x_j- y)(x_j-\gamma_k) } N\rho_V(y) \rd y\Bigg| \non\\
& \le \frac C N  \sum_{  k< L-2B   }
  \frac {\gamma_{k+1}-\gamma_k } {(x_j-\gamma_k)^2 } \non\\
&\le  \frac C B   +\frac{C}{N}\sum_{k\le \kappa N/2} |\gamma_{k+1}-\gamma_k|\le \frac{C}{B} +\frac{C}{N}
\non
\end{align}
using that $x_j\ge y_L\ge \gamma_L-\delta \ge \gamma_{L-2B} +cBN^{-1}$.
The estimate of $\Omega_j^{2,2}$ is trivial
$$
  |\Omega_j^{2,2}|\le \frac{|\gamma_{L-2B}-y_{L-2B}|}{cBN^{-1}}\le \frac{CN\delta}{B}.
$$
Finally
 \begin{align}\label{613new}
  \E_{\mu_\by^{B,\tau} }    \sum_{L \le j \le L + 4B}  [\Omega^{2,3}_j ]^2
 & \le    \E_{\mu_\by^{B,\tau} }   \sum_{L \le j \le L + 2B}   \left [  \frac 1 N  \sum_{k < L-2B }
  \Big(\frac 1 {x_j- y_k} -     \frac 1 {x_j- \gamma_k}\Big)\right ] ^2
\non\\
& \le  \frac {C\delta^2}  {N^2}   \sum_{L \le j  \le L +  4B }
      \E_{\mu_\by^{B,\tau} }   \sum_{\kappa N/2\le  k < L-2B }         \frac { 1   } { (x_j- y_k)^2  }
 \sum_ {\kappa N/2\le a < L-2B}   \frac { 1   } {   (x_j- \gamma_{a})^2}
\non\\
& + \frac{C}{N^2}\sum_{k\le \kappa N/2}|\gamma_k-y_k|^2
\non\\
&  \le C\delta^2   \frac { N^2} { B} +  C KN^{-2/5+ \varphi},
 \non
\end{align}
where we again split the summation over $k$ into $\kappa N/2\le k\le L-2B$ and
$k\le \kappa N/2$, yielding the two terms, similarly to \eqref{frr} and \eqref{6141}.

The estimate $\Omega_j^{3,low}$ is analogous to that of $\Omega_j^{2,1}$.
The first term of $\Omega_j^4$ is estimated as before in \eqref{Om4}.
The additional second term in $\Omega_j^4$ is trivial
by recalling $|\gamma_{L-2B}-\th'_{L-2B}|\le CB^2N^{-2}+C\delta$ from \eqref{38}:
$$
\Bigg|  \int_{y_{L-2B}}^{ \th'_{L-2B} }    \frac {\rho_{W_s}(y)  } {x_j- y } \rd y\Bigg|
 \le \frac{|y_{L-2B}-\th'_{L-2B}|}{cBN^{-1}}\le \frac{C\delta + CB^2N^{-2}}{cBN^{-1}}
\le\frac{C\delta N}{B} + \frac{CB}{N},
$$
since the denominator can be estimated by using
$x_j-y_{L-2B}\ge y_L-y_{L-2B}\ge \gamma_L-\gamma_{L-2B}- 2\delta\ge cBN^{-1}$
and
$$
 x_j-\th'_{L-2B}\ge y_L-\th'_{L-2B} \ge \gamma_L-\gamma_{L-2B} - 2\delta +(\gamma_{L-2B}-
\th'_{L-2B})\ge cBN^{-1}
$$
where we used $\delta \ll BN^{-1}$ and $B\ll N$.

Collecting all the error terms into \eqref{largebound}
and removing some redundant terms, we
 have thus proved Lemma \ref{eb}. \qed

\bigskip 

\noindent{\bf Acknowledgement.}  We would like to thank M. Ledoux 
for pointing out an error in the statement of Lemma \ref{lem:localLSI}
in an early  version of this paper.

\end{document}